\documentclass[10pt,reqno]{amsart}

\usepackage{amssymb}
\usepackage{mathrsfs}
\usepackage{longtable}

\usepackage{lunadiagrams}

\newtheorem{theorem}{Theorem}[subsection]
\newtheorem{lemma}[theorem]{Lemma}
\newtheorem{proposition}[theorem]{Proposition}

\newtheorem{definition}[theorem]{Definition}

\numberwithin{equation}{section}

\newcommand{\defect}{\mathrm{d}}

\newcommand{\C}{\mathbb C}
\newcommand{\Z}{\mathbb Z}
\newcommand{\IN}{\mathbb N}

\newcommand{\A}{\mathbf A}

\newcommand{\s}{\mathscr S}

\newcommand{\card}{\mathrm{card}}

\newcommand{\GL}{\mathrm{GL}}
\newcommand{\SL}{\mathrm{SL}}
\newcommand{\SO}{\mathrm{SO}}
\newcommand{\Sp}{\mathrm{Sp}}

\DeclareMathOperator{\Lie}{Lie}
\DeclareMathOperator{\supp}{supp}
\DeclareMathOperator{\rank}{rank}

\begin{document}
\title{Primitive wonderful varieties}
\author{P.\ Bravi and G.\ Pezzini}
\begin{abstract}
We complete the classification of wonderful varieties initiated by D.~Luna. We review the results that reduce the problem to the family of primitive varieties, and report the references where some of them have already been studied. Finally, we analyze the rest case-by-case.
\end{abstract}

\maketitle

\section*{Introduction}

Let $G$ be a reductive connected linear algebraic group over the field of complex numbers $\C$. 
A {\em wonderful $G$-variety} is a complete and smooth $G$-variety with remarkable properties 
(see Definition~\ref{def:wonderful}), generalizing the wonderful compactifications of symmetric spaces 
defined by C.~De Concini and C.~Procesi in \cite{DP83}.

In the article \cite{Lu01} D.~Luna started a research program to classify wonderful varieties by means of certain invariants called {\em spherical systems}, which can be represented as combinatorial objects attached to the Dynkin diagram of $G$. A strategy to prove the classification, also known as the {\em Luna conjecture}, consists in reducing the problem to a distinguished class of cases called {\em primitive} (see \cite[Section 4.2]{Lu01}). This approach was already used in \cite{Lu01}, where groups $G$ of semisimple type $\mathsf A$ were considered, and in other works: \cite{BP05}, \cite{B07}, \cite{BC10}.

In this paper we complete the proof of the Luna conjecture along the lines of this program.  Thanks to \cite[Theorem 2]{Lu01} and the Luna-Vust theory of embeddings of spherical homogeneous spaces (see \cite{Kn91}), this completes the classification of {\em spherical varieties}, i.e.\ normal $G$-varieties where a Borel subgroup of $G$ has a dense orbit. 

Another proof of the Luna conjecture, with different methods, has been proposed by S.~Cupit-Foutou in \cite{CF}.

Our work is based on the original strategy of \cite{Lu01}; we also apply some additional techniques developed in our previous article \cite{BP11}, which lead to an updated and more restrictive definition of primitive wonderful varieties and primitive spherical systems, see Definition~\ref{def:primitive}. The combinatorial properties of this notion have been already discussed by the first-named author in \cite{B09}, where a complete list of primitive spherical systems is obtained.

It is already known that the invariants we consider distinguish between different $G$-isomorphism classes of wonderful varieties (see \cite{Lo09}), therefore we achieve the classification proving that each primitive system is {\em geometrically realizable}, i.e.\ comes from a wonderful variety.

The artice is organized as follows. In Section~\ref{s:reduction} we review the known results that lead to the definition of primitive spherical systems. 
This is done briefly, except for the notion of {\em decomposable} systems (see Section~\ref{s:products}) where a more detailed discussion is needed.

Then we discuss in Sections~\ref{s:L} and \ref{sect:expl} all cases of \cite{B09}. Some of them are already well-known, for example those corresponding to reductive wonderful subgroups of $G$ (see \cite{BP11b}). We refer for brevity a few other known cases to existing publications, and we analyze in detail the remaining ones.

A relevant byproduct of this proof of the Luna conjecture is an explicit description, albeit laborious, of a generic stabilizer of a wonderful variety using only its spherical system.

Indeed, if a wonderful $G$-variety $X$ is not primitive, or admits a so-called {\em quotient of higher defect} (see Definition~\ref{def:combqhd}), or has a {\em tail} (see Definition~\ref{def:tails}) then the results in \cite{BP11} provide a description of a generic stabilizer $H\subset G$ of $X$. The description of $H$ is concise, and relates $H$ to generic stabilizers of those varieties that can be considered the ``primitive components'' of $X$. If $X$ is primitive without quotients of higher defect and without tails then we describe the subgroup $H$ in this paper, referring ultimately to explicit lists.

\subsubsection*{Acknowledgments}
The second-named author was supported by the DFG Schwerpunktprogramm 1388 -- Darstellungstheorie.

\section{Classification of wonderful varieties}

\subsection{Notations}
All algebraic groups and varieties are defined over the field $\C$ of complex numbers. We fix a connected reductive algebraic group $G$, a maximal torus $T$ and a Borel subgroup $B\supset T$ of $G$. We denote by $S$ the corresponding set of simple roots, and by $B_-$ the opposite Borel subgroup of $B$ with respect to $T$. In a connected Dynkin diagram we will denote by $\alpha_1,\alpha_2,\ldots$ or $\alpha'_1,\alpha'_2,\ldots$ the simple roots, numbered as in Bourbaki. If $G$ is semisimple then the fundamental dominant weights will be denoted by $\omega_1, \omega_2, \ldots$ and numbered as the simple roots.

\subsection{Definitions and statement of the main result}
We start collecting some definitions and basic facts on wonderful and spherical varieties. In this and in the next sections we refer to \cite{Lu01} for details and further references.

\begin{definition}\label{def:wonderful}
A $G$-variety $X$ is {\em wonderful} (of {\em rank} $r$) if it is complete and non-singular, with an open $G$-orbit whose complement is the union of $r$ non-singular prime $G$-divisors $D_1,\ldots,D_r$, any subset of these prime $G$-divisors has a transversal and non-empty intersection, and these intersections are exactly all $G$-orbit closures of $X$.
\end{definition}

A wonderful $G$-variety is spherical (see \cite{Lu96}), i.e.\ it is normal with a dense $B$-orbit, and it is the unique (up to $G$-equivariant isomorphisms) wonderful compactification of its open $G$-orbit. The radical of $G$ is known to act trivially on it, therefore we may assume that the group $G$ is semi-simple.

The {\em spherical system} of a wonderful $G$-variety $X$ is a triple $\s_X=(S^p_X, \Sigma_X, \A_X)$ of invariants of $X$, defined as follows. 

Let $P_X$ be the stabilizer of the open $B$-orbit of $X$; it is a parabolic subgroup of $G$ containing $B$. Then $S^p_X$ denotes the subset of simple roots spanning  the root system of the standard Levi subgroup of $P_X$. 

By $\Sigma_X$ we denote the set of {\em spherical roots} of $X$, i.e.\ the $T$-weights of the quotient of tangent spaces $T_zX/T_z(Gz)$, where $z$ is the unique point of $X$ fixed by $B_-$. The set $\Sigma_X$ is a basis of the lattice of $B$-eigenvalues of $B$-eigenvectors in $\C(X)$. 

The {\em colors} of $X$ are its $B$-stable and not $G$-stable prime divisors, and a color $D$ is {\em moved} by a simple root $\alpha\in S$ if $D$ is not stable under the minimal parabolic subgroup of $G$ containing $B$ and corresponding to $\alpha$. By $\Delta_X$ we denote the set of colors of $X$ and by $\Delta_X(\alpha)$ the set of colors moved by $\alpha$. By \cite[Proposition 3.2]{Lu01} a simple root moves at most two colors of $X$ and it moves exactly two colors if and only if it is a spherical root. The set of colors of $X$ is a disjoint union of three subsets $\Delta_X=\Delta_X^a\cup\Delta_X^{2a}\cup\Delta_X^b$ where:
\begin{itemize}
\item $\Delta_X^a=\bigcup\Delta_X(\alpha)$ for all $\alpha\in S\cap\Sigma_X$,
\item $\Delta_X^{2a}=\bigcup\Delta_X(\alpha)$ for all $\alpha\in S\cap{\frac12}\Sigma_X$, 
\item $\Delta_X^b=\bigcup\Delta_X(\alpha)$ for all $\alpha\in S\smallsetminus(S_X^p\cup\Sigma_X\cup{\frac12}\Sigma_X)$.
\end{itemize} 
Furthermore, if $\alpha$ and $\beta$ belong to $S\smallsetminus(S_X^p\cup\Sigma_X)$, $\Delta_X(\alpha)=\Delta_X(\beta)$ only if $\alpha$ and $\beta$ are orthogonal and $\alpha+\beta\in\Sigma_X\cup2\Sigma_X$.

There is a $\Z$-bilinear pairing $c_X\colon \Z\Delta_X\times\Z\Sigma_X\to\Z$, called full Cartan pairing, induced by the valuations of $B$-stable divisors on $B$-eigenvectors in $\C(X)$. 
The set $\Delta_X^a$ is also denoted by $\A_X$ as well as $\Delta_X(\alpha)$ is denoted by $\A_X(\alpha)$ if the simple root $\alpha$ belongs to $\Sigma_X$. 
The full Cartan pairing is uniquely determined by its restriction to $c_X\colon \Z\A_X\times\Z\Sigma_X\to\Z$,
called restricted Cartan pairing,
indeed 
\[c_X(D,\sigma)=\left\{\begin{array}{ll}
{\frac12}\langle\alpha^\vee,\sigma\rangle & \mbox{ if $D\in\Delta_X(\alpha)$ with $\alpha\in S\cap{\frac12}\Sigma_X$} \\
\langle\alpha^\vee,\sigma\rangle & \mbox{ if $D\in\Delta_X(\alpha)$ with $\alpha\in S\smallsetminus(S_X^p\cup\Sigma_X\cup{\frac12}\Sigma_X)$} 
\end{array}\right.\]

The triple $\s_X=(S^p_X, \Sigma_X, \A_X)$ is a spherical $G$-system in the sense of the following definition.

\begin{definition}\label{def:system}
Let $(S^p,\Sigma,\A)$ be a triple such that $S^p \subset S$, $\Sigma$ is a linearly independent set of characters of $B$, and $\A$ a finite set endowed with a $\Z$-bilinear pairing $c\colon \Z\A\times\Z\Sigma\to\Z$. For every $\alpha \in \Sigma \cap S$, let $\A (\alpha)$ denote the set $\{D \in \A \colon c(D,\alpha)=1 \}$. 
Such a triple is called a {\em spherical $G$-system} (of {\em rank} $r=|\Sigma|$) if the following conditions are satisfied. 
\begin{itemize} 
\item[(A1)] For every $D \in \A$ we have $c(D,-)\leq 1$, and if $c(D, \sigma)=1$ for some $\sigma\in\Sigma$ then $\sigma \in S\cap\Sigma$. 
\item[(A2)] For every $\alpha \in \Sigma \cap S$ the set $\A(\alpha)$ contains exactly two elements; denoting with $D_\alpha^+$ and $D_\alpha^-$ these elements, it holds $c(D_\alpha^+,\sigma) + c(D_\alpha^-,\sigma) = \langle \alpha^\vee , \sigma \rangle$ for all $\sigma\in\Sigma$.
\item[(A3)] The set $\A$ is the union of $\A(\alpha)$ for all $\alpha\in\Sigma \cap S$.
\item[($\Sigma 1$)] If $2\alpha \in \Sigma \cap 2S$ then $\frac{1}{2}\langle\alpha^\vee, \sigma \rangle$ is a non-positive integer for every $\sigma \in \Sigma \smallsetminus \{ 2\alpha \}$.
\item[($\Sigma 2$)] If $\alpha, \beta \in S$ are orthogonal and $\alpha + \beta$ belongs to $\Sigma$ or $2\Sigma$ then $\langle \alpha ^\vee , \sigma \rangle = \langle \beta ^\vee , \sigma \rangle$ for every $\sigma \in \Sigma$.
\item[(S)] For every $\sigma \in \Sigma$, there exists a wonderful $G$-variety $X$ of rank $1$ with $S^p_X=S^p$ and $\Sigma_X=\{\sigma\}$. 
\end{itemize}
\end{definition}

We notice that the above definition is purely combinatorial, since wonderful varieties of rank $1$ are classified (see \cite{Br89}). For an explicit list of all spherical roots appearing in wonderful varieties we refer to \cite{W}. An equivalent combinatorial version of axiom (S) can be found in \cite[Section~1.1.6]{BL11}.

A spherical system $\s$ is {\em geometrically realizable} if it is of the form $\s=\s_X$ for a wonderful variety $X$.

The classification of wonderful varieties is then given by the following.

\begin{theorem}\label{thm:luna}
The map $X\mapsto \s_X$ induces a bijection between $G$-isomorphism classes of wonderful $G$-varieties and spherical $G$-systems.
\end{theorem}

The injectivity of the map of the above theorem has been proven in \cite{Lo09}. The rest of this paper is devoted to the proof of the surjectivity.

A general observation about the proof: since spherical systems and wonderful varieties of rank $\leq 2$ are known to fulfill Theorem~\ref{thm:luna} (see  \cite{A,Br89,W}), we will prove geometric realizability of any given spherical system assuming that all systems of strictly lower rank have this property.

\subsection{Spherical closure}

The results in \cite{Lu01} reduce the classification to a certain class of wonderful varieties and spherical systems, both called {\em spherically closed}. This is actually negligible in our proof of Theorem~\ref{thm:luna}, except for the fact that it will allow us to assume that $G$ is adjoint. We review briefly this reduction, and refer to \cite[\S1, \S2, \S6 and \S7]{Lu01} and also \cite[\S2.4]{BL11} for details and proofs.

\begin{definition}
A subgroup $H\subseteq G$ is {\em spherical} if $G/H$ is a spherical $G$-variety, it is {\em wonderful} if $G/H$ admits a wonderful completion.
\end{definition}

Let $H\subseteq G$ be a wonderful subgroup and $X$ the wonderful completion of $G/H$. The normalizer $N_GH$ of $H$ in $G$ acts naturally by $G$-equivariant automorphisms on $G/H$; this induces an action of $N_GH$ on the set of colors of $X$, and the kernel of this action is called the {\em spherical closure} $\overline H$ of $H$.

In \cite{Kn96} it is shown that the homogeneous space $G/\overline H$ admits a wonderful completion $Y$, called the {\em spherical closure} of $X$, and $X$ is called {\em spherically closed} if $H=\overline H$ and $X=Y$.

The {\em spherical closure} $\overline \s$ of a spherical system $\s$ is obtained from $\s$ by replacing $\sigma$ with $2\sigma$ for all $\sigma$ spherical root not belonging to the root lattice or having the following form:
\begin{itemize}
\item[$(\mathsf B$)] $\sigma = \alpha_1+\ldots+\alpha_n$, 
where $\{\alpha_1,\ldots,\alpha_n\}\subseteq S$ has type $\mathsf B_n$ and $\alpha_n\in S^p$, or
\item[$(\mathsf G$)] $\sigma = 2\alpha_1+\alpha_2$, where $\{\alpha_1,\alpha_2\}\subseteq S$ has type $\mathsf G_2$.
\end{itemize}
A spherical system $\s$ is called {\em spherically closed} if $\overline \s=\s$. The geometric and the combinatorial definitions agree: the spherical system of the spherical closure of $X$ is $\overline{\s_X}$. This fact is shown in \cite[Section 7]{Lu01} assuming Theorem~\ref{thm:luna}, but it is also an immediate consequence of \cite[Theorem 2]{Lo09}.

A way to consider the relationship between $\s$ and $\overline\s$ is expressed by the fact that the group $\Z\Sigma$ and the Cartan pairing $c$ of $\s$ are an {\em augmentation} (see \cite[\S2.2]{Lu01}) of the spherical system $\overline \s$, and $(\overline\s, \Z\Sigma, c)$ is a {\em homogeneous spherical datum} of $G$ in the sense of \cite[\S2.2]{Lu01}.

In general, these notions are defined as follows: a couple $(\Xi, c)$ is an {\em augmentation} of a spherical system $\s=(S^p,\Sigma,\A)$ if $\Xi$ is a group of characters of $B$ containing $\Z\Sigma$ and $c$ is a pairing $c\colon \Z\A\times \Xi\to \Z$ extending the Cartan pairing of $\s$, such that the axioms (A2), ($\Sigma$1) and ($\Sigma$2) hold for all $\sigma\in \Xi$, except for the non-positivity condition of axiom ($\Sigma$1). An {\em homogeneous spherical datum} of $G$ is a triple $(\s,\Xi,c)$ where $\s$ is a spherical $G$-system and $(\Xi,c)$ is an augmentation of $\s$.

These constructions are used in \cite{Lu01} to deal with more general spherical varieties than wonderful ones, so not all homogeneous spherical data are obtained as above, i.e.\ not all are of the form $(\overline\s, \Z\Sigma, c)$ where $\s$ is a spherical system. Let us call here {\em special} those that have such form.

It is equivalent to work with a spherical $G$-system $\s$ or with $(\overline \s, \Z\Sigma, c)$, since one determines the other. Therefore we can restate Theorem~\ref{thm:luna} equivalently as follows:
\begin{quote} The map $X\mapsto (\overline{\s_X}, \Z\Sigma_X, c_X)$ induces a bijection between $G$-isomorphism classes of wonderful $G$-varieties and special homogeneous spherical data of $G$.
\end{quote}

Finally, let $Z(G)$ be the center of $G$ and $G_0=G/Z(G)$ the adjoint group of $G$. Thanks to \cite[Theorem 2]{Lu01}, the above second version of Theorem~\ref{thm:luna} is a consequence of the following statement:
\begin{quote} The map $X\mapsto \s_X$ induces a bijection between $G_0$-isomorphism classes of spherically closed wonderful $G_0$-varieties and spherically closed spherical $G_0$-systems.
\end{quote}

Therefore we will assume from now on that $G$ is an adjoint group. Under this assumption all spherical roots belong to the root lattice. Actually, they are sums of simple roots, i.e.\ $\Sigma\subseteq \IN S$.

\subsection{Colors and quotients}
Let $\s=(S^p,\Sigma,\A)$ be a spherical $G$-system. We define its set of {\em colors}, in analogy with the fact that $\A_X$ is not the full set of colors of a wonderful variety $X$. We accomplish this task combinatorially, using the fact that $\A_X$ is precisely the set of colors moved by simple roots that move two colors each. Therefore the other simple roots move at most one color each, and their behaviour is governed by \cite[Proposition 3.2]{Lu01}.

\begin{definition}
The set of {\em colors} of $\s$ is the finite set $\Delta$ obtained as disjoint union $\Delta=\Delta^a\cup\Delta^{2a}\cup\Delta^b$ where:
\begin{itemize}
\item $\Delta^a=\A$,
\item $\Delta^{2a}=\{D_\alpha\;|\; \alpha\in S\cap{\frac1 2}\Sigma\}$,
\item $\Delta^b=\{D_\alpha \;|\; \alpha\in S\smallsetminus(S^p\cup\Sigma\cup{\frac1 2}\Sigma)\}/\sim$, where $D_\alpha\sim D_\beta$ if $\alpha$ and $\beta$ are orthogonal and $\alpha+\beta\in\Sigma$.
\end{itemize} 
For all $\alpha\in S$ the set of colors {\em moved by $\alpha$} is:
\[\Delta(\alpha)=\left\{\begin{array}{ll}
\varnothing & \mbox{ if $\alpha\in S^p$} \\
\A(\alpha) & \mbox{ if $\alpha\in\Sigma$} \\
\{D_\alpha\} & \mbox{ otherwise}
\end{array}\right.\]
The (full) Cartan pairing of $\s$ is the $\Z$-bilinear map $c\colon\Z\Delta\times\Z\Sigma\to\Z$ defined as:
\[c(D,\sigma)=\left\{\begin{array}{ll}
c(D,\sigma) & \mbox{ if $D\in\Delta^a$} \\
{\frac1  2}\langle\alpha^\vee,\sigma\rangle & \mbox{ if $D=D_\alpha\in\Delta^{2a}$} \\
\langle\alpha^\vee,\sigma\rangle & \mbox{ if $D=D_\alpha\in\Delta^b$} 
\end{array}\right.\]
\end{definition} 

Therefore, if $X$ is a wonderful $G$-variety, the set of colors of $\s_X$ is naturally identified with the set of colors of $X$.

This allows to define {\em quotients} of spherical systems, i.e.\ the combinatorial counterpart of a certain class of morphisms between wonderful varieties. Let $\s=(S^p,\Sigma,\A)$ be a spherical $G$-system with set of colors $\Delta$ and Cartan pairing $c$.

\begin{definition}
A subset of colors $\Delta'\subset\Delta$ is distinguished if for all $D\in\Delta'$ there exists a positive integer $a_D$ such that $\sum_{D\in\Delta'}a_D c(D,\sigma)\geq0$ for all $\sigma\in\Sigma$.
\end{definition}

\begin{proposition}[{\cite[Theorem~3.1]{B09}}]
If $\Delta'\subset\Delta$ is distinguished then:
\begin{itemize}
\item the monoid $\{\sigma\in\mathbb N\Sigma\;|\; c(D,\sigma)=0\mbox{ for all }D\in\Delta'\}$ is free;
\item setting $S^p/\Delta'=\{\alpha\;|\; \Delta(\alpha)\subset \Delta'\}$, $\Sigma/\Delta'$ equal to the basis of the above monoid and $\A/\Delta'=\cup_{\alpha\in S\cap\Sigma/\Delta'}\A(\alpha)$, the triple $(S^p/\Delta', \Sigma/\Delta', \A/\Delta')$ is a spherical $G$-system.
\end{itemize}
\end{proposition}

If $\Delta'$ is distinguished then the spherical $G$-system $\s/\Delta'=(S^p/\Delta', \Sigma/\Delta', \A/\Delta')$ is called the {\em quotient} of $\s$ by $\Delta'$. We also use the notation $\s\to\s/\Delta'$. The set of colors of $\s/\Delta'$ can be identified with $\Delta\smallsetminus\Delta'$.

On the geometric side, let now $f\colon X\to Y$ be a surjective $G$-morphism with connected fibers between wonderful $G$-varieties. Then the subset $\Delta_f=\{D\in\Delta_X\;|\; f(D)=Y\}$ is distinguished, and $\s_Y=\s_X/\Delta_f$. Moreover, the following holds.

\begin{proposition}[{\cite[\S3.3]{Lu01}}]\label{prop:morphisms}
Let $X$ be a wonderful $G$-variety. The map $f\mapsto \Delta_f$ induces a bijection between distinguished subsets of $\Delta_X$ and $G$-isomorphism classes%
\footnote{Here two morphisms $f_1\colon X\to Y_1$, $f_2\colon X\to Y_2$ are {\em $G$-isomorphic} if there is a $G$-equivariant isomorphism $\varphi\colon Y_1\to Y_2$ such that $f_2=\varphi\circ f_1$.}
of surjective $G$-morphisms with connected fibers from $X$ onto another wonderful $G$-variety.
\end{proposition}

\section{Reduction to the primitive cases}\label{s:reduction}
We reduce the proof of Theorem~\ref{thm:luna} to a certain subclass of wonderful varieties and spherical systems called primitive, via four main reduction techniques: {\em localization}, {\em decomposition into fiber product}, {\em positive combs} and {\em tails}.

\subsection{Localizations}
For all $\sigma=\sum n_\alpha \alpha$ in $\mathbb N S$, set $\supp\sigma=\{\alpha\in S \;|\; n_\alpha\neq0\}$. For all $\Sigma\subset \mathbb N S$, set $\supp\Sigma=\cup_{\sigma\in\Sigma}\supp\sigma$.

The reduction step discussed in this section applies to all spherical systems where the set $\supp\Sigma$ does not cover the whole $S$.

\begin{definition}
Let $\s=(S^p,\Sigma,\A)$ be a spherical $G$-system. For all subsets of simple roots $S'\subseteq S$, consider a semi-simple group $G_{S'}$ with set of simple roots $S'$; we define the {\em localization} $\s_{S'}$ of $\s$ in $S'$ as the spherical $G_{S'}$-system $((S')^p,\Sigma',\A')$ as follows:
\begin{itemize}
\item $(S')^p=S^p\cap S'$,
\item $\Sigma'=\{\sigma\in\Sigma\;|\; \supp\sigma\subseteq S'\}$,
\item $\A'=\cup_{\alpha\in S\cap\Sigma'}\A(\alpha)$.
\end{itemize}
\end{definition}

The geometric counterpart of the above definition is the following.
\begin{definition}
Let $X$ be a wonderful $G$-variety. For all subsets of simple roots $S'\subseteq S$ we define the {\em localization} $X_{S'}$ of $X$ in $S'$ to be the subvariety $X^{P^r}$ of points fixed by the radical $P^r$ of $P$, where $P$ is the parabolic subgroup containing $B_-$ and corresponding to $S'$.
\end{definition}

\begin{proposition}[{\cite[\S3.2]{Lu01}}]\label{prop:locS}
Under the action of $G_{S'}=P/P^r$ the variety $X_{S'}$ is wonderful, and 
\[\s_{(X_{S'})}=(\s_X)_{S'}.\]
\end{proposition}

We come to the actual reduction step.

\begin{proposition}[{\cite[\S3.4]{Lu01}}]\label{prop:induction}
Let $\s=(S^p,\Sigma,\A)$ be a spherical $G$-system and let $S'\subseteq S$ be a subset containing $S^p\cup\supp\Sigma$. If there exists a wonderful $G_{S'}$-variety $Y$ with spherical system $\s_{S'}$ then there exists a wonderful $G$-variety $X$ with spherical system $\s$. Precisely, $X$ is a {\em parabolic induction} of $Y$, that is, 
\[X\cong G\times^P Y\] where $P$ is the parabolic subgroup containing $B_-$ corresponding to $S'$, and we let $P$ act on $Y$ via its quotient $P/P^r=G_{S'}$.
\end{proposition}

Let $\s=(S^p,\Sigma,\A)$ be a spherical $G$-system with $S^p\cup\supp\Sigma=S$, then $\supp\Sigma$ and $S^p\smallsetminus\supp\Sigma$ are orthogonal, thus $G\cong G_{\supp\Sigma}\times G_{S^p\smallsetminus\supp\Sigma}$ and the second factor acts trivially on any wonderful $G$-variety $X$ with spherical system $\s$. Therefore, the previous proposition implies the following.

\begin{proposition}\label{prop:cuspidal}
Let $\s=(S^p,\Sigma,\A)$ be a spherical $G$-system and let $S'$ be a subset of simple roots containing $\supp\Sigma$. If there exists a wonderful $G_{S'}$-variety with spherical system $\s_{S'}$ then there exists a wonderful $G$-variety with spherical system $\s$. 
\end{proposition}

\begin{definition}
A spherical $G$-system $\s=(S^p,\Sigma,\A)$ is called {\em cuspidal} if $\supp\Sigma=S$.
\end{definition}

We end this section with another kind of localization, which corresponds in the geometrical setting to taking an irreducible $G$-stable subvariety of a wonderful $G$-variety. We refer again to \cite{Lu01} for details and proofs.

\begin{definition}
Let $\s=(S^p,\Sigma,\A_{\Sigma})$ be a spherical $G$-system and $\Sigma'$ a subset of $\Sigma$. The {\em localization} of $\s$ in $\Sigma'$ is the spherical system $\s_{\Sigma'}=(S^p,\Sigma',\A_{\Sigma'})$, where
\[
\A_{\Sigma'} = \bigcup_{\alpha\in S\cap\Sigma'} \Delta(\alpha)
\]
and the Cartan pairing of $\s_{\Sigma'}$ is the one of $\s$ restricted to $\Z\Sigma'$.
\end{definition}

\begin{proposition}[{\cite[\S3.2]{Lu01}}]\label{prop:locSigma}
Let $X$ be a wonderful $G$-variety, and $\Sigma'\subseteq \Sigma_X$. Then $X$ has a unique irreducible $G$-stable closed subvariety $Y$ (that is then automatically wonderful) whose set of spherical roots is $\Sigma'$. Moreover $\s_Y=\s_{\Sigma'}$. 
\end{proposition}

For later reference, we also recall how morphisms behave with respect to localizations in a subset of spherical roots. Let $\s=(S^p,\Sigma,\A_{\Sigma})$ be a spherical system and $\Sigma'\subseteq\Sigma$. Given a subset $\widetilde\Delta$ of the colors $\Delta$ of $\s$, we define the {\em restriction} of $\widetilde\Delta$ to $\s_{\Sigma'}$ in the following way (see also \cite[\S3.2]{BP11}).

Let us denote by $\Delta^b_1\subseteq\Delta^b$ the set of colors moved by only one simple root, and by $\Delta^b_2$ the set $\Delta^b\smallsetminus\Delta_1^b$. We also write $\Delta_{\Sigma'}$ for the whole set of colors of the spherical system $\s_{\Sigma'}$, and correspondingly $\Delta_{\Sigma'}(\alpha)$ denotes the colors of $\s_{\Sigma'}$ moved by $\alpha\in S$ (notice that $\Delta_{\Sigma'}(\alpha)$ is identified with $\Delta(\alpha)$ if $\alpha\in\Sigma'\cap S$).

Then we define the restriction of $\widetilde\Delta$ to $\s_{\Sigma'}$ as
\[
\widetilde\Delta|_{\Sigma'} = \left(\widetilde\Delta_{\Sigma',1}\right)\cup\left(\widetilde\Delta_{\Sigma',2}\right)\cup\left(\widetilde\Delta_{\Sigma',3}\right)\cup\left(\widetilde\Delta_{\Sigma',4}\right),
\]
where:
\[
\widetilde\Delta_{\Sigma',1} =  \left(\bigcup_{\alpha\in\Sigma'\cap S} \Delta(\alpha)\cap\widetilde\Delta \right) \cup \left(\bigcup_{\begin{array}{c}\scriptstyle{\alpha\in(\Sigma\smallsetminus\Sigma')\cap S }\\  \scriptstyle{\textrm{with } \Delta(\alpha)\subseteq \widetilde\Delta}\end{array}} \Delta_{\Sigma'}(\alpha)\right),
\]
\[
\widetilde\Delta_{\Sigma',2} =  \bigcup_{\begin{array}{c}\scriptstyle{\alpha\in\frac12\Sigma\cap S }\\  \scriptstyle{\textrm{with } \Delta(\alpha)\subseteq \widetilde\Delta}\end{array}} \Delta_{\Sigma'}(\alpha),
\]
\[
\widetilde\Delta_{\Sigma',3} =  \bigcup_{\begin{array}{c}\scriptstyle{\{\alpha\}\in\Delta^b_1 }\\  \scriptstyle{\textrm{with } \Delta(\alpha)\subseteq \widetilde\Delta}\end{array}} \Delta_{\Sigma'}(\alpha),
\]
\[
\widetilde\Delta_{\Sigma',4} =  \bigcup_{\begin{array}{c}\scriptstyle{\{\alpha,\beta\}\in\Delta^b_2 }\\  \scriptstyle{\textrm{with } \Delta(\alpha)=\Delta(\beta)\subseteq \widetilde\Delta}\end{array}} \left(\Delta_{\Sigma'}(\alpha)\cup\Delta_{\Sigma'}(\beta)\right).
\]

With this definition, taking quotients is compatible with localization in $\Sigma'$, as stated by the following lemma.
\begin{lemma}\cite[Lemma 2.6.1]{BP11}
Let $\widetilde\Delta$ be a distinguished set of colors of $\s$. 
Then the restriction $\widetilde\Delta|_{\Sigma'}$ is a distinguished set of colors of $\s_{\Sigma'}$, and we have
\begin{equation}
(\s/\widetilde\Delta)_{\Sigma''} = (\s_{\Sigma'}) / (\widetilde\Delta|_{\Sigma'}),
\end{equation}
where $\Sigma''$ consists of the elements of $\Sigma/\widetilde\Delta$ that are linear combinations of elements of $\Sigma'$.
\end{lemma}

\subsection{Decompositions into fiber product}\label{s:products}

After \cite{BP11} was published, D.~Luna pointed out a mistake in Section~4 therein: 
the first combinatorial condition of \cite[Definition~4.1.2]{BP11}, i.e.\ the definition of {\em decomposition} of a spherical system, is not equivalent to the equation (4) in \cite[\S4]{BP11}. The consequence is that \cite[Definition~4.1.2]{BP11} is not the exact combinatorial counterpart of the condition that a wonderful variety is a fiber product.

Here we correct that mistake by slightly revising \cite[Definition~4.1.2]{BP11}. We also show that all this does not affect our proof of Theorem~\ref{thm:luna} based on the lists of \cite[Theorems~2.10 and~2.12]{B09}, although they had been established using \cite[Definition~4.1.2]{BP11}.

Let $X_1$, $X_2$ and $X_3$ be wonderful $G$-varieties 
and let $\varphi_1\colon X_1\to X_3$ and $\varphi_2\colon X_2\to X_3$ be 
surjective equivariant maps with connected fibers. 
The fiber product $X=X_1\times_{X_3}X_2$ is a $G$-variety, 
with $\psi_1\colon X\to X_1$, $\psi_2\colon X\to X_2$ surjective equivariant maps with connected fibers 
such that $\varphi_1\circ\psi_1=\varphi_2\circ\psi_2$. It is not necessarily spherical. 

Let $Z_1$, $Z_2$ and $Z_3$ be the respective closed $G$-orbits, with the corresponding restricted maps $\varphi_1$ and $\varphi_2$. 

\begin{definition}[{\cite[Definition~4.1.1]{BP11}}]\label{def:prodgeom}
The $G$-variety $X_1\times_{X_3}X_2$ is called a {\em wonderful fiber product} 
if it is wonderful with closed $G$-orbit $Z_1\times_{Z_3}Z_2$.
\end{definition}

\begin{definition}\label{def:prod}
Let $\s=(S^p,\Sigma,\A)$ be a spherical $G$-system with set of colors $\Delta$. 
Two distinguished sets of colors $\Delta_1,\Delta_2$ are said to {\em decompose} $\s$ if:
\begin{enumerate}
\item \label{def:prod:Sp}
\begin{enumerate}
\item $S^p/\Delta_1 \cap S^p/\Delta_2$ is equal to $S^p$,
\item every connected component of $S^p/(\Delta_1\cup\Delta_2)$ is contained in either $S^p/\Delta_1$ or $S^p/\Delta_2$,
\end{enumerate}
\item \label{def:prod:Sigma} $\Sigma$ is included in $\Sigma/\Delta_1 \cup \Sigma/\Delta_2$.
\end{enumerate}
If $\s$ admits two non-empty such subsets of colors then $\s$ is called {\em decomposable}.
\end{definition}

We report now some useful statements that follow immediately from the above definition. Under the assumption (2) of Definition~\ref{def:prod} 
\begin{itemize}
\item[(P1)] there exist no $\sigma\in\Sigma$, $D_1\in\Delta_1$, $D_2\in\Delta_2$ 
such that both $c(D_1,\sigma)$ and $c(D_2,\sigma)$ are non-zero. 
\end{itemize}
This together with the assumption (1a) implies that 
\begin{itemize}
\item[(P2)] $\Delta_1$ and $\Delta_2$ are disjoint.
\end{itemize}
Moreover, $\Delta_2$ (resp.\ $\Delta_1$) is a distinguished set of colors of $\s/\Delta_1$ (resp.\ $\s/\Delta_2$), 
$\Delta_1\cup\Delta_2$ is a distinguished set of colors of $\s$ 
and 
\begin{itemize}
\item[(P3)] $(\s/\Delta_1)/\Delta_2=(\s/\Delta_2)/\Delta_1=\s/(\Delta_1\cup\Delta_2)$.  
\end{itemize}
In particular, 
\begin{itemize}
\item[(P4)] $S^p/(\Delta_1\cup\Delta_2)$ is disjoint union of $S^p$, $S^p/\Delta_1\smallsetminus S^p$ and $S^p/\Delta_2\smallsetminus S^p$. 
\end{itemize}

Before discussing the relationship between Definitions~\ref{def:prodgeom} and \ref{def:prod}, we recall the local structure of wonderful varieties (see e.g.\ \cite[Theorem 2.3]{Kn94}). Consider a wonderful $G$-variety $Y$ and the open subset $M_Y$ obtained removing from $Y$ all its colors. Recall that $Y$ is {\em toroidal}, i.e.\ no $G$-orbit is contained in a color. This implies that $M_Y$ is the union of all $B$-orbits of $Y$ that are dense in their respective $G$-orbits.

Now $M_Y$ has a closed subvariety isomorphic to $\C^{\Sigma_Y}$ such that the standard Levi subgroup $L$ of $P_Y$ containing $T$ acts linearly on $\C^{\Sigma_Y}$ via the spherical roots of $Y$, and the map
\[
P_Y^u\times \C^{\Sigma_Y} \to M_Y
\]
induced by the action is an isomorphism.

\begin{lemma}\label{lemma:local}
Let $f\colon Y\to Y'$ be a surjective morphism with connected fibers between two wonderful varieties. Then $f(M_Y)=M_{Y'}$, and there is a choice of the subvariety $\C^{\Sigma_Y}$ of $M_Y$ such that $f(\C^{\Sigma_Y})=\C^{\Sigma_{Y'}}$. Moreover, identifying $P_Y^u\times\{0\}$ (resp.\ $P_{Y'}^u\times\{0\}$) with its image in $M_Y$ (resp.\ $M_{Y'}$), we have $f(P_Y^u\times\{0\})=P_{Y'}^u\times\{0\}$.
\end{lemma}
\begin{proof}
The fact that $f(M_Y)=M_{Y'}$ follows from the fact that $Y$ and $Y'$ have finitely many $B$-orbits and that $M_Y$ (resp.\ $M_Y'$) is the union of all $B$-orbits of $Y$ (resp.\ $M_{Y'}$) that are dense in their respective $G$-orbits.

Define
\[
M_Y' = Y\smallsetminus \bigcup_{D\in \Delta_{Y'}} f^{-1}(D).
\]
Then the map
\[
P_{Y'}^u\times Z\to  M_Y'
\] 
induced by the action is an isomorphism, and $Z=f^{-1}(\C^{\Sigma_{Y'}})$ is a spherical $L_{Y'}$-variety where $L_{Y'}$ is the standard Levi subgroup of $P_{Y'}$ containing $T$ (see also \cite[Remark 3.5.3]{Lo09}).

Define now $M_Z$ as $Z$ without all its colors (with respect to the action of $L_{Y'}$ and its Borel subgroup $L_{Y'}\cap B$). Again by \cite[Theorem 2.3]{Kn94} there exists a closed subvariety $V$ of $M_Z$ such that the map
\[
(P_Y^u\cap L_{Y'})\times V \to M_Z
\]
induced by the product is an isomorphism. Since $M_Y$ is the image of $P_{Y'}^u\times M_Z$, we can choose $V$ as the subvariety $\C^{\Sigma_Y}$ of $M_Y$, so that $f(\C^{\Sigma_Y})\subseteq \C^{\Sigma_{Y'}}$. In addition $f(\C^{\Sigma_Y})$ intersects all $G$-orbits of $Y'$, thus all $T$-orbits of $\C^{\Sigma_{Y'}}$, therefore $f(\C^{\Sigma_Y})=\C^{\Sigma_{Y'}}$.

The last equality is true because $P_{Y}^u\times\{0\}$ is the intersection of $M_Y$ with the closed $G$-orbit of $Y$, and the same holds in $Y'$.
\end{proof}

\begin{proposition}\label{prop:products}\
\begin{enumerate}
\item\label{prop:products:productdecomposes}
Let $X_1$, $X_2$ and $X_3$ be wonderful $G$-varieties 
with $\varphi_1\colon X_1\to X_3$ and $\varphi_2\colon X_2\to X_3$ 
surjective equivariant maps with connected fibers.
If $X=X_1\times_{X_3}X_2$ is a wonderful fiber product, and denoting $\psi_1\colon X\to X_1$ and $\psi_2\colon X\to X_2$ the corresponding maps, then the distinguished sets of colors $\Delta_{\psi_1}$ and $\Delta_{\psi_2}$ decompose $\s_X$.
\item\label{prop:products:decomposedisproduct}
Let $\s=(S^p,\Sigma,\A)$ be a spherical $G$-system with set of colors $\Delta$, 
and let $\Delta_1,\Delta_2$ be distinguished subsets of $\Delta$ that decompose $\s$. 
If $X_1$ and $X_2$ are wonderful $G$-varieties 
with spherical system $\s/\Delta_1$ and $\s/\Delta_2$, respectively,
with their surjective equivariant maps with connected fibers 
onto a wonderful $G$-variety $X_3$ with spherical system $\s/(\Delta_1\cup\Delta_2)$,
then the $G$-variety $X_1\times_{X_3}X_2$ is a wonderful fiber product with spherical system $\s$.
\end{enumerate}
\end{proposition}

\begin{proof}
We prove part (\ref{prop:products:productdecomposes}). First observe that the subset of colors of $X$ mapped dominantly onto $X_3$ is $\Delta_{\psi_1}\cup\Delta_{\psi_2}$.

The closed $G$-orbit of $X$ is $G/P_{X}=Z_1\times_{Z_3}Z_2$ where $Z_i=G/P_{X_i}$ is the closed $G$-orbit of $X_i$ for $i\in\{1,2,3\}$.  Comparing stabilizers of the base points we deduce that $P_X=P_{X_1}\cap P_{X_2}$, therefore $S^p/\Delta_{\psi_1}\cap S^p/\Delta_{\psi_2}=S^p_X$.

From the fact that $Z_1\times_{Z_3}Z_2$ is a single $G$-orbit we deduce that $P_{X_3}=P_{X_1}P_{X_2}$, and this implies that $L_3=L_1L_2$ where for all $i\in\{1,2,3\}$ we denote by $L_i$ the standard Levi subgroup of $P_{X_i}$ containing $T$. This is only possible if each simple factor of the commutator of $L_3$ is contained either in $L_1$ or in $L_2$, which translates into the condition that every connected component of $S^p/(\Delta_{\psi_1}\cup\Delta_{\psi_2})$ is contained in either $S^p/\Delta_{\psi_1}$ or $S^p/\Delta_{\psi_2}$.

Property (\ref{def:prod:Sp}) of Definition~\ref{def:prod} is proven, let us show property (\ref{def:prod:Sigma}). Thanks to Lemma~\ref{lemma:local} the open chart $M_X$ of the local structure of $X$ is $M_{X_1}\times_{M_{X_3}}M_{X_2}$, and we have
\[
\C^{\Sigma_X} = \C^{\Sigma_{X_1}}\times_{\C^{\Sigma_{X_3}}}\C^{\Sigma_{X_2}},
\]
which is equivalent to
\[
\C[\IN\Sigma]=\C[\IN\Sigma_{X_1}]\otimes_{\C[\IN\Sigma_{X_3}]}\C[\IN\Sigma_{X_2}].
\]
This implies $\Sigma\subseteq \Sigma_{X_1}\cup\Sigma_{X_2}$, which is the desired inclusion.

We prove part (\ref{prop:products:decomposedisproduct}). 
For $i\in\{1,2\}$ write $\Sigma/\Delta_i=\Sigma_i^b\cup\Sigma_i'$ 
where $\Sigma_i^b=(\Sigma/\Delta_i)\cap\Sigma$ and $\Sigma_i'=(\Sigma/\Delta_i)\smallsetminus\Sigma$. 
Then we have $\Sigma_i'\subseteq\IN(\Sigma\smallsetminus\Sigma_i^b)$, 
which is a consequence of the fact that $\IN(\Sigma/\Delta_i)$ is by definition 
equal to the intersection of $\IN\Sigma$ with the kernels of all colors of $\Delta_i$.

At this point property (\ref{def:prod:Sigma}) of Definition~\ref{def:prod} implies
\begin{equation}\label{eqn:Sigma}
\Sigma_i'\subseteq\IN(\Sigma_{3-i}^b)
\end{equation}
for all $i\in\{1,2\}$, and again by definition of $\Sigma/\Delta_i$ we have
\begin{equation}\label{eqn:intersection}
\IN(\Sigma/(\Delta_1\cup\Delta_2))=\IN(\Sigma/\Delta_1)\cap\IN(\Sigma/\Delta_2).
\end{equation}

Set $\Delta_3=\Delta_1\cup\Delta_2$.
Now consider the inclusions of rings  $\C[\IN(\Sigma/\Delta_i)]\hookrightarrow\C[\IN\Sigma]$ for all $i\in\{1,2,3\}$ and the induced natural map
\[
\varphi\colon\C[\IN(\Sigma/\Delta_1)]\otimes_{\C[\IN(\Sigma/\Delta_3)]}\C[\IN(\Sigma/\Delta_2)]\to \C[\IN\Sigma]
\]
Thanks to property (\ref{def:prod:Sigma}) of Definition~\ref{def:prod} the map $\varphi$ is surjective; using (\ref{eqn:Sigma}) and (\ref{eqn:intersection}) it is elementary to show that $\varphi$ is also injective. We conclude that there is a $T$-equivariant isomorphism
\begin{equation}\label{eqn:localstr1}
\C^\Sigma\cong\C^{\Sigma/\Delta_1}\times_{\C^{\Sigma/\Delta_3}}\C^{\Sigma/\Delta_2}.
\end{equation}

We turn now to $P^u$, where $P$ is the parabolic subgroup of $G$ containing $B$ and associated with $S^p$. Thanks to condition (\ref{def:prod:Sp}) of Definition~\ref{def:prod} we have that $P_{X_3}^u=P_{X_1}^u\cap P_{X_2}^u$, and $P^u$ is the product of its two normal subgroups $P_{X_1}^u$ and $P_{X_2}^u$. In other words
\begin{equation}\label{eqn:localstr2}
P^u\cong P_{X_1}^u\times_{P_{X_3}^u}P_{X_2}^u.
\end{equation}

Consider now the variety $X=X_1\times_{X_3}X_2$. It has an open subset equal to $M=M_{X_1}\times_{M_{X_3}}M_{X_2}$; thanks to Lemma~\ref{lemma:local} the isomorphisms (\ref{eqn:localstr1}) and (\ref{eqn:localstr2}) induce a $TP^u$-equivariant isomorphism
\begin{equation}\label{eqn:M}
M\cong P^u\times \C^\Sigma
\end{equation}
where $T$ acts by conjugation on the first factor. Notice that the standard Levi subgroup $L$ of $P$ containing $T$ is the intersection of the standard Levi subgroups of $P_{X_1}$ and of $P_{X_2}$; it follows that the isomorphism (\ref{eqn:M}) is also $P$-equivariant, where $L$ acts by conjugation on the factor $P^u$ and the commutator $(L,L)$ acts trivially on $\C^\Sigma$.

Therefore $M$ has an open $B$-orbit, and $X$ is a spherical variety.  It remains to show that $X$ is wonderful, with spherical system $\s$.

First notice that any color $D$ of $X$ doesn't intersect $M$, whence $D$ is mapped not dominantly on $X_i$ for some $i\in\{1,2\}$. As a consequence $D$ doesn't contain any $G$-orbit of $X$, because $X_1$ and $X_2$ are toroidal, in other words $X$ is toroidal too. It also follows that $M$ intersects all $G$-orbits of $X$, hence $X$ is smooth and {\em simple}, which means by definition that it has a unique closed $G$-orbit.

To sum up $X$ is a complete, smooth, simple, toroidal spherical variety: then it is wonderful (see e.g.\ \cite[Theorem 30.15]{Ti11}). The open subset of its local structure is $M$, therefore $S^p_X=S^p$ and $\Sigma_X=\Sigma$. It remains to prove that $\A_X$ can be identified with $\A$ compatibly with the Cartan pairing.

First of all Proposition~\ref{prop:morphisms} implies that there is a bijection between $\A_{X_3}$ with a subset of $\A_{X_i}$ for all $i\in\{1,2\}$, and a bijection of $\A_{X_i}$ with a subset of $\A_X$, such that these bijections preserve the property of being moved by a simple root and are compatible with the Cartan pairings. Considering that a simple root that is also a spherical root moves exactly two colors, and that $\Sigma_X\subseteq \Sigma_{X_1}\cup\Sigma_{X_2}$, we conclude that $\A_X$ is the union of $\A_{X_1}$ and $\A_{X_2}$ where we identify $D_1\in\A_{X_1}$ with $D_2\in\A_{X_2}$ if $D_1$ and $D_2$ are the same color of $\A_{X_3}$

By properties (P2) and (P3) the same is true for $\A$ in place of $\A_X$, inducing an identification of $\A$ with $\A_X$ such that $\Delta_1\cap\A$ and $\Delta_2\cap\A$ correspond to the colors of $\A_X$ mapped dominantly onto resp.\ $X_1$ and $X_2$.

Let $D$ be a color of $\A_X$ and $\sigma\in\Sigma_X$. If $D\in\A_{X_i}$ and $\sigma\in\Sigma_{X_i}$ for some $i\in\{1,2\}$, then $c(D,\sigma)=c_{X_i}(D,\sigma)$ by definition of quotient spherical system, and $c_{X_i}(D,\sigma)=c_X(D,\sigma)$ by Proposition~\ref{prop:morphisms}. The equality $c(D,\sigma)=c_X(D,\sigma)$ follows.

Finally, we have to show the same equality when $\sigma\notin\Sigma_{X_i}$, and in this case $\sigma\in\Sigma_{X_{3-i}}$, which implies that any color in $\Delta_{3-i}$ is zero on $\sigma$. Let $\alpha\in\Sigma_{X_i}$ be a simple root moving $D$; we may assume that $\alpha\notin\Sigma_{X_{3-i}}$. Then $\Delta_{3-i}$ contains some color moved by $\alpha$.

To summarize, at least one of the two colors of $\s$ moved by $\alpha$ is zero on $\sigma$, and if one has a non-zero value then it is the only one not in $\Delta_{3-i}$ and its value on $\sigma$ is $\langle\alpha^\vee,\sigma\rangle$. The same holds for the colors of $X$ moved by $\alpha$, thanks to the same argument where we replace $\Delta_i$ by the set of colors of $X$ mapped dominantly onto $X_i$. This proves $c(D,\sigma)=c_X(D,\sigma)$.
\end{proof}


We conclude the section showing that Definition~\ref{def:prod} is equivalent to \cite[Definition~4.1.2]{BP11} 
and \cite[Definition~2.1]{B09} under further mild assumptions, which are in particular satisfied by primitive spherical systems and primitive positive 1-combs (see Definition~\ref{def:primitive}) of rank $>1$.

\begin{lemma}\label{lemma:equivalent}
Let $\s=(S^p,\Sigma,\A)$ be a spherical $G$-system with set of colors $\Delta$. 
Assume it is cuspidal and of rank $>1$.
The spherical system $\s$ is decomposable if and only if
there exist distinguished subsets $\Delta_1$ and $\Delta_2$ of $\Delta$ such that
\begin{itemize}
\item[(1')] ($S^p/\Delta_1\smallsetminus S^p) \perp (S^p/\Delta_2\smallsetminus S^p$),
\item[(2)] $\Sigma\subset(\Sigma/\Delta_1\cup\Sigma/\Delta_2)$. 
\end{itemize}
\end{lemma}

\begin{proof}
One implication is easy (and no assumption on cuspidality and rank is needed). 
If $\Delta_1$ and $\Delta_2$ are distinguished subsets of $\Delta$ that decompose $\s$, 
they satisfy conditions (1a), (1b) and (2) of Definition~\ref{def:prod}. 
In particular (P4) holds, so condition (1b) implies condition (1').

Let us pass to the other implication.
Assume there exist distinguished subsets $\Delta_1$ and $\Delta_2$ of $\Delta$ 
satisfying conditions (1') and (2). 
Then the subsets $\Delta_1$ and $\Delta_2$ satisfy condition (1a), too.
Assume that they do not satisfy condition (1b), 
so there exists a string of $m\geq3$ adjacent simple roots, say $\alpha_1,\ldots,\alpha_m$,
such that $\alpha_1\in S^p/\Delta_1\smallsetminus S^p$, $\alpha_2,\ldots,\alpha_{m-1}\in S^p$
and $\alpha_m\in S^p/\Delta_2\smallsetminus S^p$.
Here we use the cuspidality assumption and obtain in particular that $\alpha_2,\ldots,\alpha_{m-1}\in\supp\Sigma$.
By axiom (S) of Definition~\ref{def:system} (see \cite[Section~1.1.6]{BL11} for an equivalent combinatorial version),
every $\alpha\in S^p$ must be orthogonal to every $\sigma\in\Sigma$. 
Therefore, in our case there exists at least one spherical root $\sigma$ such that $\alpha_2,\ldots,\alpha_{m-1}\in\supp\sigma$.
This contradicts the property (P1) unless $\sigma$ is equal to $\alpha_1+\ldots+\alpha_m$ with support of type $\mathsf B_m$,
see \cite[Table~1]{BL11} for the list of spherical roots.  
Moreover, for all $\sigma'\in\Sigma\smallsetminus\{\sigma\}$ we have $\supp\sigma'\cap\supp\sigma=\varnothing$.

We now use the rank $>1$ assumption. 
So the set $S\smallsetminus\supp\sigma$ which is equal to $\supp(\Sigma\smallsetminus\{\sigma\})$ is nonempty.

If $S\smallsetminus\supp\sigma$ is orthogonal to $\supp\sigma$ then $\s$ is decomposable: 
indeed $\Delta'_1=\{D\in\Delta(\alpha):\alpha\not\in\supp\sigma\}$ 
and $\Delta'_2=\{D\in\Delta(\alpha):\alpha\in\supp\sigma\}$ decompose $\s$.

Otherwise $S\smallsetminus\supp\sigma$ necessarily has a connected component of type $\mathsf A_r$
non-orthogonal to $\supp\sigma$. 
By \cite[Lemma~2.22(2)]{B09} there exists a non-empty distinguished set $\Delta'_1$ of colors 
moved only by simple roots in $S\smallsetminus\supp\sigma$ 
and such that $c(D,\sigma)=0$ for all $D\in\Delta_1$.  
Set $\Delta'_2$ consisting of only one color, 
the unique one moved by a simple root in $\supp\sigma$ and orthogonal to $S\smallsetminus\supp\sigma$.    
The sets $\Delta'_1$ and $\Delta'_2$ decompose $\s$.
\end{proof}

\subsection{Positive combs}\label{s:poscomb}
In this reduction step we deal with spherical systems having a color $D\in \A$ such that $c(D,-)$ is positive on all spherical roots. In loose terms, the geometric realizability of such a system depends on the geometric realizability of other systems that contain the same kind of color, but where the latter is moved only by one simple root.

\begin{definition}
Let $\s=(S^p,\Sigma,\A)$ be a spherical $G$-system. A {\em positive comb} of $\s$ is an element $D$ of $\A$ such that $c(D,\sigma)\geq0$ for all $\sigma\in\Sigma$. It is also called {\em positive $n$-comb} if $n=\card\{\alpha\in S\cap\Sigma \;|\; c(D,\alpha)=1\}$.
\end{definition}

Let $\s=(S^p,\Sigma,\A)$ be a spherical $G$-system with a positive comb $D$. Set $S_D=\{\alpha\in S\cap\Sigma\;|\; c(D,\alpha)=1\}$. For all $\alpha\in S_D$, define $\s_\alpha=(S^p,\Sigma_\alpha,\A_\alpha)$ where $\Sigma_\alpha=\Sigma\smallsetminus(S_D\smallsetminus\{\alpha\})$ and $\A_\alpha=\cup_{\beta\in S\cap\Sigma_\alpha}\A(\beta)$. Then the spherical $G$-system $\s_\alpha$ has a positive $1$-comb in $\A_\alpha(\alpha)$.

\begin{proposition}[{\cite[\S6]{BP11}}]\label{prop:ncomb}
Let $\s=(S^p,\Sigma,\A)$ be a spherical $G$-system with a positive $n$-comb $D$, with $n>1$. If for all $\alpha\in S_D$ there exists a wonderful $G$-variety with spherical system $\s_\alpha$, then there exists a wonderful $G$-variety with spherical system $\s$.
\end{proposition}

In this case a principal isotropy group of the wonderful $G$-variety with spherical system $\s$ can be explicitly constructed starting from the principal isotropy groups of the wonderful $G$-varieties with spherical systems $\s_\alpha$, for $\alpha\in S_D$ (see \cite[\S5.3 and \S5.4]{BP11} for details).

We will recall how this construction works in a more general case, in Section~\ref{s:mqhd}.

\subsection{Tails}

In this section we recall some results of \cite{BP11}, which show that in proving geometric realizability of a spherical system $\s$ it is possible to ``remove'' certain kinds of spherical roots, called {\em tails}, together with some of the simple roots of their support.

\begin{definition}\label{def:tails}
Let $\s=(S^p,\Sigma,\A)$ be a spherical $G$-system. A {\em tail} of $\s$ is a subset of spherical roots $\widetilde\Sigma\subset\Sigma$ with $\supp\widetilde\Sigma$ included in a connected component $S_0=\{\alpha_1,\ldots,\alpha_n\}$ of $S$ such that:\begin{enumerate}
\item there exists a distinguished subset of colors $\Delta'$ of $\s$ with $\Sigma/\Delta'=\widetilde\Sigma$, and
\item one of the following cases occur:
\begin{itemize}
\item {\em (type $b(m)$)} $S_0$ is of type $\mathsf B_n$, $1\leq m\leq n$, $\widetilde\Sigma=\{\alpha_{n-m+1}+\ldots+\alpha_n\}$ and $\alpha_n\in S^p$ if $m>1$ or $c(D^+_{\alpha_n},\sigma')=c(D^-_{\alpha_n},\sigma')$ for all $\sigma'\in\Sigma$ if $m=1$; 
\item {\em (type $2b(m)$)} $S_0$ is of type $\mathsf B_n$, $1\leq m\leq n$, and $\widetilde\Sigma=\{2\alpha_{n-m+1}+\ldots+2\alpha_n\}$; 
\item {\em (type $c(m)$)} $S_0$ is of type $\mathsf C_n$, $2\leq m\leq n$, and $\widetilde\Sigma=\{\alpha_{n-m+1}+2\alpha_{n-m+2}+\ldots+2\alpha_{n-1}+\alpha_n\}$;
\item {\em (type $d(m)$)} $S_0$ is of type $\mathsf D_n$, $2\leq m\leq n$, and $\widetilde\Sigma=\{2\alpha_{n-m+1}+\ldots+2\alpha_{n-2}+\alpha_{n-1}+\alpha_n\}$;
\item {\em (type $(aa,aa)$)} $S_0$ is of type $\mathsf E_6$ and $\widetilde\Sigma=\{\alpha_1+\alpha_6,\alpha_3+\alpha_5\}$; 
\item {\em (type $(d3,d3)$)} $S_0$ is of type $\mathsf E_7$ and $\widetilde\Sigma=\{\alpha_2+2\alpha_4+\alpha_5,\alpha_5+2\alpha_6+\alpha_7\}$;
\item {\em (type $(d5,d5)$)} $S_0$ is of type $\mathsf E_8$ and $\widetilde\Sigma=\{2\alpha_1+\alpha_2+2\alpha_3+2\alpha_4+\alpha_5,\alpha_2+\alpha_3+2\alpha_4+2\alpha_5+2\alpha_6\}$;
\item {\em (type $(2a,2a)$)} $S_0$ is of type $\mathsf F_4$ and $\widetilde\Sigma=\{2\alpha_3,2\alpha_4\}$.
\end{itemize}
\end{enumerate}
\end{definition}

\begin{proposition}[{\cite[\S6]{BP11}}]\label{prop:tails}
Let $\s=(S^p,\Sigma,\A)$ be a spherical $G$-system with a tail $\widetilde\Sigma$. Set $S'=\supp(\Sigma\smallsetminus\widetilde\Sigma)$. If there exists a wonderful $G_{S'}$-variety with spherical system $\s_{S'}$ then there exists a wonderful $G$-variety with spherical system $\s$.
\end{proposition}

In this case the principal isotropy group of the wonderful $G$-variety with spherical system $\s$ can be explicilty constructed starting from the principal isotropy group of the wonderful $G_{S'}$-variety with spherical system $\s_{S'}$ (see \cite[\S6]{BP11} for details).

Let us recall how this construction works, referring for details and proofs to \cite[\S6]{BP11}. Suppose that $\s$ is cuspidal and that $\widetilde\Sigma$ is a tail of $\s$ not of type $c(m)$, denote by $H_\ell$ a generic stabilizer for the localization of $\s$ in $S \smallsetminus \supp\widetilde\Sigma$, and by $H_q$ one corresponding to the quotient $\s/\Delta'$. Let also $Q$ be a parabolic subgroup of $G$ minimal containing $H_q$, and $R$ be a parabolic subgroup of $G_{S'}$ minimal containing $H_\ell$.

We can assume that $Q$ contains the chosen maximal torus $T$ of $G$, and we denote by $S_Q$ the set of simple roots of $G$ belonging to the standard (i.e.\ containing $T$) Levi subgroup $L_Q$ of $Q$. This Levi subgroup, up to a torus factor, is actually isomorphic to $G_{S'\cap S_Q}\times G_{\supp\widetilde\Sigma}$. Then $H_q$ can be choosen in such a way that $H_q=(G_{S'\cap S_Q}\times K_q)Q^r$, where $K_q$ is a subgroup of $G_{\supp\widetilde\Sigma}$.

Let us denote by $L_R$ the standard Levi subgroup of $R$. By construction there is a central isogeny $\pi\colon G_{S'\cap S_Q} Z(L_Q) \to L_R$. The subgroup $H_\ell$ can be chosen to have Levi subgroup $L_\ell$ inside $L_R$, and we set $L_\ell'=\pi^{-1}(L_\ell)$. We define then $L=L_\ell' \times K_q$. It is a subgroup of $H_q$.

Now the unipotent radical $H^u_\ell$ of $H_\ell$ turns out to correspond to a subgroup $U$ of the unipotent radical $Q^u$ of $Q$ in such a way that $U$ is stable under conjugation by $L$ and the two quotients $\Lie R^u/\Lie H_\ell^u$ and $\Lie Q^u/\Lie U$ are isomorphic as $L$-modules. Here we let $L$ act on $\Lie R^u$ via the map $L\to L_\ell$ given by the projection on $L_\ell'$ followed by the map $\pi$.

Finally, the subgroup $H$ corresponding to $\s$ is $H=LU$. The construction in the case of type $c(m)$ is slightly more involved, but very similar.

Let us illustrate the subgroups appearing in the above procedure in an example. Consider the following system $\s$ with a tail of type $b(m)$, for $G= \SO(2n+1)$:
\[
\begin{picture}(14700,600)(-300,-300)
\put(0,0){\usebox{\mediumam}}
\put(5400,0){\usebox{\plusbm}}
\end{picture}
\]
\[
\s=(\{ \alpha_2,\ldots,\alpha_{l-1},\alpha_{l+2},\ldots,\alpha_{l+m} \}, \{ \alpha_1+\ldots+\alpha_l, \alpha_{l+1}+\ldots+\alpha_{l+m} \}, \varnothing)
\]
where $l=n-m$. 
The system $\s$ has rank $2$, its wonderful subgroup $H$ of $G$ has Levi factor $L=\GL(l)\times \SO(2m)$ 
and unipotent radical $U$ with $\Lie U\cong \C^l\otimes\C^{2m}\oplus\wedge^2\C^l$. 
The quotient $\s/\Delta'$ has rank $1$ and following diagram
\[
\begin{picture}(14700,600)(-300,-300)
\put(0,0){\usebox{\edge}}
\put(1800,0){\usebox{\shortsusp}}
\put(3600,0){\usebox{\edge}}
\put(5400,0){\usebox{\wcircle}}
\put(5400,0){\usebox{\plusbm}}
\end{picture}
\]
Its wonderful subgroup of $G$ is $H_q=(\GL(l)\times \SO(2m))Q^u$. 
It is a parabolic induction by means of the parabolic subgroup $Q$ of $G$, where $S_Q=S\smallsetminus\{ \alpha_l \}$. Therefore $Q$ has Levi part $L_Q=\GL(l)\times\SO(2m+1)$.
The set $S'$ is $\{ \alpha_1,\ldots,\alpha_l\}$, and the localization $\s_{S'}$ is
\[
\begin{picture}(14700,600)(-300,-300)
\put(0,0){\usebox{\mediumam}}
\end{picture}
\]
whose wonderful subgroup of $\SL(l+1)$ is $H_\ell=L_\ell=\GL(l)$. 
The set of colors $\Delta'$, restricted to $\s_{S'}$, corresponds to the inclusion $\GL(l)\subset R$ 
where $R$ is the maximal parabolic subgroup of $\SL(l+1)$ corresponding to the simple root $\alpha_l$. The $L_\ell$-module $\Lie R^u/\Lie H_\ell^u$ is isomorphic to $\C^l$, and correspondingly the unipotent radical $U$ of $H$ is the subgroup of $Q^u$ such that $\Lie Q^u/\Lie U\cong\C^l$.

\subsection{Primitive cases}\label{s:primitive}
The above results lead to the following.

\begin{definition}\label{def:primitive}
\begin{itemize}
\item A spherical $G$-system is {\em primitive} if it is cuspidal, not decomposable, without positive combs and without tails. Correspondingly, a wonderful variety is {\em primitive} if its spherical system is.
\item A positive $1$-comb of a spherical $G$-system $\s$ is called {\em primitive} if $\s$ is cuspidal, not decomposable and without tails.
\end{itemize}
\end{definition}

After Propositions~\ref{prop:cuspidal}, \ref{prop:products}, \ref{prop:ncomb}, and \ref{prop:tails}, Theorem~\ref{thm:luna} holds provided that all primitive spherical systems and all spherical systems with a primitive positive 1-comb are geometrically realizable.

Wonderful varieties with rank $\leq 2$ are well known after \cite{A,Br89,W} and in that case Theorem~\ref{thm:luna} holds.

Therefore we assume that the rank is at least $3$. Then, thanks to Lemma~\ref{lemma:equivalent}, primitive spherical systems and spherical systems with a primitive positive 1-comb are classified in \cite{B09}. To complete the proof of Theorem~\ref{thm:luna}, we examine them all in the next sections: some cases are already known in the literature, some can be treated with further reduction techniques, and some have to be treated explicitly case-by-case.

\subsection{Known cases}\label{s:known}
Spherical homogeneous spaces that are affine (i.e.\ of the form $G/H$ with $H$ reductive) are also well known, see \cite{K,M,Br87}. On the other hand, the wonderful $G$-varieties $X$ whose open $G$-orbit is affine are characterized by the existence of $n_{\sigma}\geq0$ for all $\sigma\in\Sigma_X$ such that $c_X(D,\sum_{\sigma\in\Sigma_X}n_{\sigma} \sigma)>0$ for all $D\in\Delta_X$. It has been shown that all spherical systems with the above property are geometrically realizable; they are also called {\em reductive} spherical systems (see \cite{BP11b}). In the notations of \cite{B09}, they are: the entire clan $\mathsf R$, $\mathsf S$-1, $\mathsf S$-2, $\mathsf S$-3, $\mathsf S$-5, $\mathsf S$-68, $\mathsf T$-1 (with $G$ of type $\mathsf A_{2n}$), $\mathsf T$-9, $\mathsf T$-12, $\mathsf T$-15, $\mathsf T$-15', $\mathsf T$-25. 

The remaining known cases are those with $G$ having a simply-laced Dynkin diagram, see \cite{Lu01,BP05,B07}, and strict wonderful varieties, see \cite{BC10}. We recall that a wonderful variety $X$ is strict if all its isotropy groups are self-normalizing, and this is equivalent to a combinatorial condition on $\s_X$: for every $\sigma \in \Sigma_X$, there exists no wonderful $G$-variety $X'$ with $S^p_{X'}=S^p_X$ and $\Sigma_{X'}=\{2\sigma\}$. 

We end this section including, as a reference, the generic stabilizers of some of the cases of the last two families above (strict and with $G$ simply laced). Precisely, we report in Tables~\ref{table:pm} and \ref{table:a} those that cannot be described with the techniques we will introduce in Section~\ref{s:L}.

The subgroups we give in the tables are all spherical, connected and equal to their normalizers, hence they are wonderful by \cite{Kn96}. Each corresponds to the correct spherical system: one may check this fact directly in the above cited papers \cite{Lu01,BP05,B07,BC10} (except for the case $\mathsf T$-4). Another proof of this fact can be given as follows: given the subgroup $H\subseteq G$ and the spherical system $\s$ of one entry, one checks in \cite{B09} that $\s$ is uniquely determined, among the primitive spherical systems, by its quotients (different from $\s$ itself), which are all known cases. Then our explicit description of $H$ implies that $\s_X$ has the same quotients, where $X$ is the wonderful compactification of $G/H$, and we conclude that $\s_X=\s$. How to read the tables:

\begin{enumerate}
\item The cases in Table \ref{table:pm} satisfy the following property: the generic stabilizer $H$ is a parabolic subgroup of a symmetric subgroup $K$ of $G$. The sixth case is $\mathsf T$-1 for $n=r+1$ and $\mathsf T$-22 otherwise, where $r$ is the rank of $G/H$; moreover $l = \lfloor r/2\rfloor$.
\begin{table}\label{table:pm}
\caption{Known cases I}
\begin{tabular}{cccc}
Case & Type of $G$ & Semisimple type of $K$ & Semisimple type of $H$ \\
\hline
$\mathsf S$-50 &  $\mathsf A_5$ & $\mathsf C_3$ & $\mathsf A_2$ \\
$\mathsf S$-62 & $\mathsf A_{p+2}$ & $\mathsf A_{p}\times\mathsf A_1$ & $\mathsf A_{p}$ \\
$\mathsf S$-67 & $\mathsf A_{p+q+3}$ & $\mathsf A_{p+q+1}\times \mathsf A_1$ & $\mathsf A_{p}\times \mathsf A_{q}\times \mathsf A_1$ \\
$\mathsf S$-75 & $\mathsf A_{2p+q}$ & $\mathsf A_{p+q-1}\times \mathsf A_p$ & $\mathsf A_{p+q-1}\times \mathsf A_{p-1}$ \\
$\mathsf S$-76 & $\mathsf A_{2p+q-1}$ & $\mathsf A_{p+q-1}\times\mathsf A_{p-1}$ & $\mathsf A_{p+q-2}\times\mathsf A_{p-1}$ \\
$\mathsf T$-1 and $\mathsf T$-22  & $\mathsf C_n$ & $\mathsf C_{l}\times\mathsf C_{n-l}$ & $\mathsf C_{l}\times\mathsf C_{n-l-1}$ \\
$\mathsf T$-1  & $\mathsf E_6$ & $\mathsf F_4$ & $\mathsf C_3$ \\
$\mathsf T$-1  & $\mathsf F_4$ & $\mathsf B_4$ & $\mathsf A_1\times\mathsf B_2$ \\
$\mathsf T$-2 & $\mathsf D_{n}$ & $\mathsf A_{n-1}$ & $\mathsf A_{n-2}$ \\
$\mathsf T$-10 & $\mathsf D_n$ & $\mathsf D_{n-1}$ & $\mathsf A_{n-1}$ \\
$\mathsf T$-11 &  $\mathsf E_6$ & $\mathsf D_5$ & $\mathsf D_4$
\end{tabular}
\end{table}
\item The cases in Table \ref{table:a} admit a generic stabilizer $H$ contained in a parabolic subgroup $P\supset B_-$ of $G$, in such a way that $H^u\subset P^u$ and there are Levi subgroups $L\subset L_P$ of $H$ and $P$ respectively, with $L$ containing the radical of $L_P$. We give the semisimple types of $L$ and $L_P$, and the $(L,L)$-module structure of $P^u/H^u$. We denote in the table by $V(\omega_i)$ the simple $(L,L)$-module with highest weight $\omega_i$.
\begin{table}\label{table:a}
\caption{Known cases II}
\begin{tabular}{cccc}
Case & Type of $G$ & Semisimple types of $L\subset L_P$ & $P^u/H^u$ \\
\hline
$\mathsf S$-105  & $\mathsf E_7$ & $\mathsf A_3\times\mathsf A_2\subset\mathsf A_6$ & $V(\omega_1)$ \\
$\mathsf T$-1  & $\mathsf B_{2n}$ & $\mathsf C_n\subset A_{2n-1}$ & $V(\omega_1)$ \\
$\mathsf T$-3  & $\mathsf E_7$ & $\mathsf A_5\subset\mathsf D_6$ & $V(\omega_1)$ \\
$\mathsf T$-4  & $\mathsf E_6$ & $\mathsf A_4\subset\mathsf D_5$ & $V(\omega_2)$ \\
$\mathsf T$-4  & $\mathsf E_8$ & $\mathsf A_6\subset\mathsf D_7$ & $V(\omega_2)$ \\
$\mathsf T$-8  & $\mathsf E_6$ & $\mathsf A_3\times\mathsf A_1\subset\mathsf A_5$ & $V(\omega_1)$
\end{tabular}
\end{table}
\end{enumerate}

\section{Further reduction techniques}\label{s:L}
\subsection{Quotients of type ($\mathscr L$)}
Let $\s=(S^p,\Sigma,\A)$ be a primitive spherical $G$-system with set of colors $\Delta$. Typically, it admits a {\em minimal quotient $\s/\Delta'$ of type ($\mathscr L$)} which is somewhat simpler and known to be geometrically realizable. 

Let us recall that if $\s$ is geometrically realizable then a quotient $\s\to\s/\Delta'$ is called of type ($\mathscr L$) if it corresponds to an inclusion $H\subset K$ of spherical subgroups of $G$ such that $K/H$ is connected, $K$ is minimal containing $H$, Levi subgroups of $H$ and $K$ ($L_H$ and $L_K$, respectively) are equal up to their connected centers and $H^u$ is strictly contained in $K^u$. In addition, the quotient is minimal if $\Lie K^u /\Lie H^u$ is a simple $L_H$-module.

To give a combinatorial version of these properties, it is necessary to take into account the chain of inclusions $H\subseteq K\subseteq L_KQ^r\subseteq Q$ where $Q$ is a parabolic subgroup of $G$ minimal containing $K$, and the corresponding chain of distinguished subsets of colors of $\s$. The result is the following definition; for more details we refer to \cite[\S5]{BP11}.

\begin{definition}
Let $\s$ be a spherical $G$-system with colors $\Delta$, and let $\Delta'$ be a distinguished subset of colors. The quotient $\s/\Delta'$ is {\em of type ($\mathscr L$)} if there exist distinguished subsets of colors $\widetilde\Delta'$ and $\Delta''$ such that $\Delta'\subseteq \widetilde\Delta'\subseteq \Delta''$ and such that:
\begin{enumerate}
\item $\Sigma/\Delta''=\varnothing$, and $\Delta''$ is minimal having this property;
\item no simple root in the support of $\Sigma/\widetilde\Delta'$ moves a color in $\Delta\smallsetminus\Delta''$;
\item there exists a linear combination of elements in $\Sigma/\widetilde\Delta'$, with positive coefficients, that takes via the Cartan pairing non-negative values on all colors in $\Delta''\smallsetminus\widetilde\Delta'$;
\item $\widetilde\Delta'$ is minimal with the above properties.
\end{enumerate}
If in addition $\Delta'$ is minimal distinguished, then $\s/\Delta'$ is {\em minimal of type ($\mathscr L$)}.
\end{definition}

We review in the next sections some special classes of such quotients. For them the above definition together with additional analysis actually provide a way to {\em construct} the principal isotropy group $H$ of a wonderful $G$-variety with spherical system $\s$.

\subsection{Minimal quotients of higher defect}\label{s:mqhd}
The defect of a spherical system $\s=(S^p,\Sigma,\A)$ with set of colors $\Delta$ is defined as 
\[\defect(\s)=\card\,\Delta - \card\,\Sigma.\]
We recall that if $\s$ is geometrically realizable and $H$ is the principal isotropy group of a wonderful variety with spherical system $\s$, then $\defect(\s)$ equals the rank of the character group of $H$.

In this section we discuss those primitive cases $\s$ that admit a minimal quotient $\s/\Delta'$ of type ($\mathscr L$) such that $\defect(\s/\Delta')>\defect(\s)$.

We are now able to define the quotients involved in the reduction step, and state the latter.

\begin{definition}\label{def:combqhd}
Let $\s=(S^p,\Sigma,\A)$ be a spherical $G$-system with set of colors $\Delta$. Let $\Delta'$ be a distinguished subset, minimal of type ($\mathscr L$) with $k=\defect(\s/\Delta')-\defect(\s)>0$.
The quotient $\s/\Delta'$ is of {\em higher defect} if there exist $k+1$ spherical roots $\sigma_0,\ldots,\sigma_k\in\Sigma$ such that, 
if we set, for all non-empty $I\subset\{0,\ldots,k\}$,  $\Sigma_I = \Sigma\smallsetminus\{\sigma_i\;|\;i\not\in I\}$,  
$\s_I=\s_{\Sigma_I}$ and $\Delta'_I=\Delta'|_{\Sigma_I}$, we have: 
\begin{enumerate}
\item $\defect(\s_I)=\defect(\s)+k+1-|I|$,
\item $\Delta'_I$ is minimal of type ($\mathscr L$),
\item $\s_I/\Delta'_I=\s/\Delta'$.
\end{enumerate}
\end{definition}

In the above definition, we also denote for simplicity $\s_{\{i\}}$ by $\s_i$.

\begin{theorem}\label{thm:hd}\cite[Theorem 5.3.1]{BP11}
Let $\s$ be a spherically closed spherical $G$-system with a quotient of higher defect $\s/\Delta'$ as in Definition~\ref{def:combqhd}. Assume that the spherical $G$-systems $\s/\Delta'$, $\s_0,\ldots,\s_k$ are geometrically realizable, and that $\s/\Delta'$ is spherically closed. Then $\s$ is geometrically realizable.
\end{theorem}

On the list of \cite{B09} one can check that the combinatorial conditions required by Definition~\ref{def:combqhd} are satisfied by all minimal quotients $\s/\Delta'$ with $\defect(\s/\Delta')>\defect(\s)$ of any primitive spherical system $\s$.

These primitive systems are: $\mathsf S$-4, $\mathsf S$-6, $\mathsf S$-8, $\ldots$,$\mathsf S$-13, $\mathsf S$-15, $\mathsf S$-16, $\mathsf S$-17, $\mathsf S$-19,$\ldots$,$\mathsf S$-49, $\mathsf S$-51, $\mathsf S$-52, $\mathsf S$-54,$\ldots$,$\mathsf S$-60, $\mathsf S$-66, $\mathsf S$-82, $\mathsf S$-83, $\mathsf S$-84, $\mathsf S$-89,$\ldots$,$\mathsf S$-94, $\mathsf S$-97,$\ldots$,$\mathsf S$-104, $\mathsf S$-106,$\ldots$,$\mathsf S$-122, $\mathsf T_1$ (with $G$ of type $\mathsf A_{2n+1}$, $\mathsf B_{2n+1}$, $\mathsf D_{n}$, $\mathsf E_{7}$, $\mathsf E_{8}$), $\mathsf T$-3 of rank $5$ and $7$, $\mathsf T$-4 of rank $6$, $\mathsf T$-5, $\mathsf T$-6, $\mathsf T$-7, $\mathsf T$-13, $\mathsf T$-14, $\mathsf T$-16, $\mathsf T$-17.

For the proof of Theorem~\ref{thm:luna} only the following 38 cases need to be examined, since the others are known after Section~\ref{s:known}: $\mathsf S$-6, $\mathsf S$-8, $\mathsf S$-12, $\mathsf S$-13, $\mathsf S$-15, $\mathsf S$-16, $\mathsf S$-19, $\mathsf S$-21, $\mathsf S$-25,$\ldots$,$\mathsf S$-31, $\mathsf S$-36,$\ldots$,$\mathsf S$-40, $\mathsf S$-43, $\mathsf S$-44, $\mathsf S$-45, $\mathsf S$-51, $\mathsf S$-52, $\mathsf S$-59, $\mathsf S$-60, $\mathsf S$-89, $\mathsf S$-98, $\mathsf S$-99, $\mathsf S$-100, $\mathsf S$-103, $\mathsf S$-109, $\mathsf S$-113, $\mathsf S$-114, $\mathsf S$-115, $\mathsf T$-16, $\mathsf T$-17.

As an example, we give the details for the system $\mathsf S$-28, leaving to the Reader the long but elementary verification on the others. Let $\s$ be the case $\mathsf S$-28. Then $G=G_1\times G_2$ with $G_1$ of type $\mathsf A_1$ and $G_2$ of type $\mathsf B_3$. $\s=(S^p,\Sigma,\A)$ with $S^p=\varnothing$, $\Sigma=S$ and $\A=\{D_1^+, D_1^-, E_1^-, E_2^+, E_3^-\}$, with Cartan pairing
\[
\begin{array}{c|cccc}
c & \alpha_1 & \alpha'_1 & \alpha'_2 & \alpha'_3 \\
\hline
D_1^+ & 1 & 1 & -1 & 1\\
D_1^- & 1 & -1 & 1 & -1\\
E_1^- & -1 & 1 & 0 & -1\\
E_2^+ & -1 & 0 & 1 & 0\\
E_3^- & -1 & -1 & -1 & 1\\
\end{array}
\]

Set $\Delta'=\widetilde\Delta'=\{D_1^+, E_2^+\}$ and $\Delta''=\Delta'\cup\{D_1^-\}$. The former is minimal distinguished and the latter is distinguished, we have $\s/\Delta'=(\varnothing, \{\alpha_1+\alpha_2'\}, \varnothing)$ and $\s/\Delta''=(\{\alpha_2'\}, \varnothing, \varnothing)$. Definition~\ref{def:combqhd} is satisfied, indeed the simple roots in the support of $\Sigma/\Delta'$ are $\alpha_1$ and $\alpha_2'$, which do not take value $1$ on $\Delta\smallsetminus\Delta''=\{E_1^-,E_3^-\}$, and the element $\alpha_1+\alpha_2'$ of $\Sigma/\Delta'=\Sigma/\widetilde\Delta'$ takes non-negative values on all colors in $\Delta''\smallsetminus\widetilde\Delta'=\{D_1^-\}$. Hence $\s/\Delta'$ is a minimal quotient of type ($\mathscr L$).

The full set of colors of $\s$ is $\A$, and the full set of colors of $\s/\Delta'$ is $\{D_{\alpha_1}=D_{\alpha_2'}, D_{\alpha_1'}, D_{\alpha_3'} \}$; the defect of $\s$ is $1$ and the defect of $\s/\Delta'$ is $2$. Hence $k=1$, and the two spherical roots $\sigma_0$, $\sigma_1$ required by Definition~\ref{def:combqhd} are resp.\ $\alpha_1'$ and $\alpha_3'$. If $I$ is equal e.g.\ to $\{0\}$, then $\s_I=(\varnothing, \{\alpha_1,\alpha_1',\alpha_2'\}, \{D_1^+, D_1^-, E_1^-, E_2^+\})$ (where however $D_1^+$ is not moved by $\alpha_3'$), the full set of colors of $\s_I$ is $\{D_1^+, D_1^-, E_1^-, E_2^+, D_{\alpha_3'}\}$, and $\s_I$ has defect $2$. The restriction $\Delta'_I$ is $\{D_1^+, E_2^+\}$, the quotient $\s_I/\Delta'_I$ is equal to $\s/\Delta'$, and one checks as above that it is minimal of type ($\mathscr L$) for $\s_I$.

Finally, all systems in this example are spherically closed, the systems $\s/\Delta'$ and $\s_{\{0,1\}}$ are geometrically realizable because they have rank $1$, the system $\s_{\{0\}}$ and $\s_{\{1\}}$ are geometrically realizable because they are parabolic inductions of the cases resp.\ $\mathsf S$-5 (a known case) and $\mathsf S$-7 (discussed explicitly in Section~\ref{sect:expl}).

Also for all the other cases the remaining assumptions of Theorem~\ref{thm:hd} are fulfilled: the geometric realizability of $\s/\Delta'$, $\s_0,\ldots,\s_k$ follows via the previous reduction techniques from cases that are either known after Section~\ref{s:known}, or discussed explicitly in Section~\ref{sect:expl}, or are primitive cases with minimal quotients of higher defect and having rank less than the rank of $\s$. Therefore, we have the following.

\begin{proposition}
Theorem~\ref{thm:luna} follows from the geometric realizability of the primitive spherical systems without minimal quotients of higher defect.
\end{proposition} 

Finally, let us give more details on how the spherical subgroup $H$ corresponding to $\s$ is related to $H_0,\ldots,H_k$ and $K$, corresponding resp.\ to $\s_0,\ldots,\s_k$ and $\s/\Delta'$.

We show in \cite[\S5]{BP11} that $H_0,\ldots,H_k$ and $K$ can be choosen so that the following holds.
\begin{enumerate}
\item A Levi factor $L_K$ of $K$ contains a subgroup $L'$ equal to $L_K$ up to the connected center (i.e.\ $L'Z(L_K)^\circ=L_K$) and such that $L'$ is also equal to a Levi subgroup of $H_i$ up to the connected center;
\item there exist an $L_K$-module decomposition $\Lie\,K^u=V\oplus W_0\oplus\ldots\oplus W_k$ and $Z(L_K)^\circ$-characters $\gamma_0,\ldots,\gamma_k$ such that for all $i,j\in \{0,\ldots,k\}$
\begin{enumerate}
\item $W_i$ is a simple module under the action of $L'$;
\item $W_i\cong W_j$ as $L'$-modules;
\item $Z(L_K)^\circ$ acts on $W_i$ via the weight $\gamma_i$;
\item $\Lie\,H^u_i = V\oplus W_0\oplus\ldots W_{i-1}\oplus W_{i+1}\oplus\ldots\oplus W_k$.
\end{enumerate}
\end{enumerate}

Then, the subgroup $H$ can be defined as a subgroup of $K$ satisfying the following:
\begin{enumerate}
\item it has a Levi subgroup $L$ equal to $L_K$ up to the connected center;
\item the connected center $Z(L)^\circ$ is the connected subgroup of $Z(L_K)^\circ$ defined by the equations $\gamma_0=\ldots=\gamma_k$;
\item $\Lie\,H^u$ is a co-simple $L$-submodule of $\Lie\, K^u$ containing $V$ but not any direct summand $W_0,\ldots,W_k$ of $\Lie\, K^u$.
\end{enumerate}

If $\s$ has a positive $k$-comb $D$ as in Section~\ref{s:poscomb} with $k>1$, then the subset of colors $\Delta'=\{D\}$ gives a quotient of higher defect. The localizations $\s_i$ have a positive comb corresponding to $D$, but in their case it is a $1$-comb; in the quotient $\s/\Delta'$ the comb $D$ is obviously not present.

The description of $H$ follows the above recipe, with the additional property that for all $i$ the submodule $W_i$ has dimension $1$. 

\subsection{Minimal quotients of rank 0}\label{s:minquotrkz}
Some primitive spherical systems $\s$ of defect $1$ admit a rank $0$ (i.e.\ with $\Sigma=\varnothing$) spherical system $\s/\Delta'$ as quotient of type ($\mathscr L$) of {\em constant defect} (i.e.\ with $\defect(\s)=\defect(\s/\Delta')$). They are: $\mathsf S$-53, $\mathsf S$-73, $\mathsf T$-23, and $\mathsf T$-26. In these cases we also have $\defect(\s)=\defect(\s/\Delta') = 1$.

Rank $0$ spherical systems correspond to partial flag varieties, namely the corresponding principal isotropy groups are parabolic subgroups, which are maximal in $G$ if the defect is equal to $1$. More precisely, such a spherical system has only one color, say $D_\alpha$ where $S\smallsetminus S^p=\{\alpha\}$: then up to conjugation we have as principal isotropy group the parabolic subgroup $Q$ containing $B_-$ corresponding to $S\smallsetminus\{\alpha\}$. The Lie algebra of the unipotent radical of $Q$ decomposes under the action of the standard Levi subgroup $L_Q\supset T$ as
\[\Lie Q^u\cong V(-\alpha)\oplus[\Lie Q^u,\Lie Q^u],\]
where $V(-\alpha)$ is the simple $L_Q$-module of highest weight $-\alpha$. This leads to a unique possible candidate $H$ for the principal isotropy group of a wonderful variety with spherical system $\s$. Namely, the group $H$ has unipotent radical $(Q^u,Q^u)$ and Levi subgroup $L_Q$.

With this choice, $H$ is a spherical subgroup of $G$ and it is equal to its normalizer, hence it is the principal isotropy group of a wonderful variety $X$ (see \cite[Corollary 7.2]{Kn96}). The only proper subgroup of $G$ strictly containing $H$ is $Q$, whence the spherical system $\s_X$ is primitive and admits $\s/\Delta'$ as a quotient of type ($\mathscr L$) of constant defect. Finally, one can check on the list in \cite{B09} that, for the four above systems, these properties identify $\s$ uniquely, so $\s=\s_X$.

Therefore, we have the following.

\begin{proposition}\label{prop:rank0quot}
If $\s$ admits a quotient spherical system $\s/\Delta'$ of type ($\mathscr L$) of constant defect of rank $0$ then it is geometrically realizable and, with the above notation, we have $H=H^u\,L$ where $L=L_Q$ and $H^u=(Q^u,Q^u)$.
\end{proposition}

\subsection{Localizations}\label{sect:loc}
Finally, we report the following obvious consequence of Proposition~\ref{prop:locS}.

\begin{proposition}\label{prop:loc}
Let $\s$ be a sperical $G$-system and $S'$ a subset of $S$. Then the geometric realizability of its localization $\s_{S'}$ follows from the geometric realizability of $\s$.
\end{proposition}

This may be also considered as a reduction technique and applied to those primitive spherical systems that are localizations of other primitive systems. Here some care is needed, since unlike the others this reduction technique works ``backwards'' with respect to the rank. To avoid any problem, we apply Proposition~\ref{prop:loc} to a case of the form $\s_{S'}$ only if the geometric realizability of $\s$ does not depend (possibly through other reduction techniques) on systems of lower rank.

This can be done in the cases: $\mathsf S$-64 which is a localization of $\mathsf S$-70 (discussed explicitly in \S\ref{sect:expl}), $\mathsf S$-65  which is a localization of $\mathsf S$-72 (discussed explicitly in \S\ref{sect:expl}), $\mathsf S$-74 and $\mathsf S$-87 which are localizations of $\mathsf S$-73 (where Proposition \ref{prop:rank0quot} applies), $\mathsf S$-95  which is a localization of $\mathsf S$-96 (discussed explicitly in \S\ref{sect:expl}), $\mathsf T$-24  which is a localization of $\mathsf T$-25 (a reductive, hence known case), $\mathsf T$-29 which is a localization of $\mathsf T$-26 (where Proposition \ref{prop:rank0quot} applies).

\section{Explicit computations}\label{sect:expl}  
In this section we study all the remaining primitive spherical systems. To be precise we are left with the primitive spherical $G$-systems $\s$ such that:
\begin{itemize}
\item the rank is $>2$,
\item $\s$ is not reductive, i.e.\ there does not exist a linear combination $\sum_{\sigma\in\Sigma}n_{\sigma} \sigma$ with $n_{\sigma}\geq0$ for all $\sigma\in\Sigma$ such that $c(D,\sum_{\sigma\in\Sigma}n_{\sigma} \sigma)>0$ for all $D\in\Delta$,
\item the Dynkin diagram of $G$ is not simply-laced,
\item $\s$ is not strict, i.e.\ there exists $\sigma \in \Sigma$ and a wonderful $G$-variety $X$ with $S^p_X=S^p$ and $\Sigma_X=\{2\sigma\}$,
\item $\s$ has no minimal quotient of higher defect,
\item $\s$ has no quotient of type ($\mathscr L$) with constant defect and rank $0$,
\item $\s$ is not one of the localizations listed in \S\ref{sect:loc}.
\end{itemize} 
They consist of 24 cases, which in the notation of \cite{B09} are: $\mathsf{ab}^\mathsf{y}(p-1,p)$, $\mathsf{ag}^\mathsf{y}(1,2)$, $\mathsf{b}^\mathsf{y}(4)$, $\mathsf{b}^\mathsf{w}(3)$, $\mathsf S$-63, $\mathsf S$-69,\ldots,$\mathsf S$-72, $\mathsf S$-77,\ldots,$\mathsf S$-81, $\mathsf S$-85, $\mathsf S$-86, $\mathsf S$-88, $\mathsf S$-96, $\mathsf T$-18,\ldots,$\mathsf T$-21, $\mathsf T$-27 and $\mathsf T$-28.

\subsection{Non-essential quotients of type ($\mathscr L$)}\label{subsect:noness}
Let $\s=(S^p,\Sigma,\A)$ be a spherical $G$-system. 

\begin{definition}
A minimal quotient $\s\to\s/\Delta'$ is called {\em essential} if 
\[(\Sigma/\Delta')\cap\Sigma=\varnothing.\]
\end{definition}

Among the $24$ cases we have to consider, the following admit a non-essential minimal quotient of type ($\mathscr L$) and of constant defect: $\mathsf S$-69, $\mathsf S$-71, $\mathsf S$-77,$\ldots$,$\mathsf S$-80, $\mathsf S$-86, $\mathsf S$-88, $\mathsf T$-18, $\mathsf T$-19, $\mathsf T$-20, $\mathsf T$-27. We show here how to treat them with a common approach: this leads to their geometric realizability with an explicit description of their principal isotropies, and only requires a trivial check on each respective non-essential quotient.

Let in general $\s\to\s/\Delta'$ be a non-essential minimal quotient of type ($\mathscr L$) of constant defect. Roughly speaking, the subset of spherical roots $(\Sigma/\Delta')\cap\Sigma$ plays no role in the co-connected inclusion $H\subset K$ corresponding to the quotient $\s\to\s/\Delta'$. 

For all $D\in\Delta'$ and $\sigma\in(\Sigma/\Delta')\cap\Sigma$, one clearly has $c(D,\sigma)=0$. Therefore, the subset $\Delta'$ can be identified with a distinguished subset of the spherical $G$-system $\widehat{\s}=(S^p,\widehat\Sigma,\widehat\A)$ with $\widehat\Sigma=\Sigma\smallsetminus(\Sigma/\Delta')$ and $\widehat\A=\cup_{\alpha\in S\cap\widehat\Sigma}\A(\alpha)$. The quotient $\widehat{\s}\to\widehat{\s}/\Delta'$ is still minimal, of type ($\mathscr L$) and of constant defect, but clearly essential.

We suppose that $\widehat\s$ and $\s/\Delta'$ are geometrically realizable, and let $\widehat H\subset\widehat K$ be the co-connected inclusion corresponding to $\widehat{\s}\to\widehat{\s}/\Delta'$. Then $\widehat H^u\subset\widehat K^u$ and $W=\Lie \widehat K^u/\Lie \widehat H^u$ is a simple $\widehat H$-module. More explicitly, we can fix the same Levi subgroup $\widehat L$ for $\widehat H$ and $\widehat K$, and $W$ is a simple spherical $\widehat L$-module. 
Let also $K$ be the general isotropy group corresponding to $\s/\Delta'$, with Levi factor $L$ and unipotent radical $K^u$.

Now the crucial claim is that $K$ and $\widehat L$ have a common direct factor $M$ that acts non-trivially on $W$, and that $K^u$ has an $L$-stable subgroup $H^u$ such that $\Lie K^u/\Lie H^u$ is $M$-isomorphic to $W$. Then we have a natural choice of $H$, namely $H=LH^u$.

We conjecture that the claim holds in general; however, it can be checked directly on the $12$ cases above. Indeed, for all of them the spherical system $\widehat{\s}$ is obtained via parabolic induction (see Proposition~\ref{prop:induction}) from a primitive system among those treated in \S\ref{s:known}, \S\ref{s:mqhd} and \S\ref{s:minquotrkz}. As a consequence $\widehat \s$ is geometrically realizable and the general isotropy $\widehat K$ is known. Moreover the quotient $\s/\Delta'$ has rank $1$ or $2$, which implies that $K$ is also known.

For each of the $12$ above cases the resulting $H$ is spherical and self-normalizing, thus wonderful. Let $X$ be the wonderful $G$-variety with general isotropy $H$: it remains to show that $\s = \s_X$.

For this, one has to compute the minimal distinguished subsets of colors of $\s$. The corresponding minimal quotients $\s_1,\s_2,\ldots$ are all of rank $1$ or rank $2$, or are obtained from a case of rank $2$ by adding a tail. As a consequence they are geometrically realizable, and the general isotropies $H_1,H_2,\ldots$ of the associated wonderful varieties are known. We underline that if $\s_i$ has rank $1$ or $2$ then $H_i$ is found in \cite{W}, but if $\s_i$ has a tail the description of $H_i$ is more involved and must be derived from \cite[\S6]{BP11}. However this occurs only once, for $\s$ equal to the case $\mathsf S$-69, and will be discussed in more details later.

One checks that $H_i$ contains $H$ up to conjugation for all $i$, and that there is no subgroup strictly between $H_i$ and $H$. This ensures that $\s_X$ also has the quotients $\s_1,\s_2,\ldots$, and the corresponding subsets of colors of $\s_X$ are minimal distinguished. Then it is elementary to show that these quotients, together with its defect, uniquely determine $\s$ among all spherical $G$-systems, which finally shows $\s_X=\s$.

To illustrate the whole procedure let us discuss the first case (where $\s$ is $\mathsf S$-69) in details. Here and in the following we make use of Luna diagrams, see \cite{B09} for their definition.

The case $\mathsf S$-69 has the following Luna diagram:
\[
\begin{picture}(16000,3000)(-300,-1500)
\put(0,0){\usebox{\aone}}
\put(0,0){\usebox{\edge}}
\put(1800,0){\usebox{\shortam}}
\put(5400,0){\usebox{\edge}}
\put(7200,0){\usebox{\edge}}
\put(9000,0){\usebox{\shortam}}
\put(12600,0){\usebox{\leftbiedge}}
\put(7200,0){\usebox{\aone}}
\put(14400,0){\usebox{\aone}}
\multiput(0,900)(7200,0){2}{\line(0,1){450}}
\put(0,1350){\line(1,0){7200}}
\multiput(7200,-900)(7200,0){2}{\line(0,-1){450}}
\put(7200,-1350){\line(1,0){7200}}
\put(0,600){\usebox{\toe}}
\put(7200,600){\usebox{\tow}}
\put(0,-900){\line(0,-1){900}}
\put(0,-1800){\line(1,0){15000}}
\put(15000,-1800){\line(0,1){2400}}
\put(15000,600){\line(-1,0){300}}
\end{picture}
\]
for $G=\Sp(2p+2q+6)$, where $p,q\geq 1$ and the set of spherical roots is $\Sigma = \{\alpha_1,\alpha_2+\ldots+\alpha_{p+1}, \alpha_{p+2},\alpha_{p+3}+\ldots+\alpha_{p+q+2},\alpha_{p+q+3}\}$. It is convenient to denote here by $Q$ the parabolic subgroup containing $B_-$ and associated to the simple roots $\alpha_1,\ldots,\alpha_p, \alpha_{p+2},\ldots,\alpha_{p+q+1}, \alpha_{p+q+3}$. Let $L$ be the Levi subgroup of $Q$ containing $T$; its commutator subgroup $L'$ is isomorphic to $\SL(p+1)\times\SL(q+1)\times\SL(2)$. The Lie algebra of the unipotent radical $Q^u$ is $L'$-isomorphic to $(\C^{p+1}\otimes\C^{q+1})\oplus(\C^{p+1}\otimes\C^{2})\oplus(\C^{q+1}\otimes\C^{2})\oplus(\C^{p+1}\otimes(\C^{q+1})^*)\oplus S^2\C^{p+1}\oplus S^2\C^{q+1}$.

The non-essential minimal quotient of constant defect $\s/\Delta'$ has Luna diagram: 
\[
\begin{picture}(16000,3000)(-300,-1500)
\put(0,0){\usebox{\edge}}
\put(1800,0){\usebox{\susp}}
\put(5400,0){\usebox{\edge}}
\put(7200,0){\usebox{\edge}}
\put(9000,0){\usebox{\shortam}}
\put(12600,0){\usebox{\leftbiedge}}
\put(5400,0){\usebox{\wcircle}}
\put(7200,0){\usebox{\wcircle}}
\put(14400,0){\usebox{\wcircle}}
\multiput(7200,-300)(7200,0){2}{\line(0,-1){450}}
\put(7200,-750){\line(1,0){7200}}
\end{picture}
\]
The latter has rank $2$, and is obtained by parabolic induction as in Proposition~\ref{prop:induction} with $S'=\{\alpha_1,\ldots,\alpha_{p},\alpha_{p+2},\ldots, \alpha_{p+q+3}\}$ and $\s'$ equal to the primitive system
\[
\begin{picture}(16000,3000)(-300,-1500)
\put(0,0){\usebox{\susp}}
\put(7200,0){\usebox{\edge}}
\put(9000,0){\usebox{\shortam}}
\put(12600,0){\usebox{\leftbiedge}}
\put(7200,0){\usebox{\wcircle}}
\put(14400,0){\usebox{\wcircle}}
\multiput(7200,-300)(7200,0){2}{\line(0,-1){450}}
\put(7200,-750){\line(1,0){7200}}
\end{picture}
\]
With the above notations, the group $K$ (associated to $\s/\Delta'$) has Levi subgroup equal to $L$, and $\Lie K^u$ is obtained from $Q^u$ by removing the summand $\C^{q+1}\otimes\C^{2}$.

The system $\widehat \s$ has Luna diagram
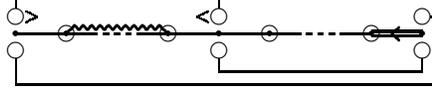
\[
\begin{picture}(16000,3000)(-300,-1500)
\put(0,0){\usebox{\aone}}
\put(0,0){\usebox{\edge}}
\put(1800,0){\usebox{\shortam}}
\put(5400,0){\usebox{\edge}}
\put(7200,0){\usebox{\edge}}
\put(9000,0){\usebox{\susp}}
\put(9000,0){\usebox{\wcircle}}
\put(12600,0){\usebox{\wcircle}}
\put(12600,0){\usebox{\leftbiedge}}
\put(7200,0){\usebox{\aone}}
\put(14400,0){\usebox{\aone}}
\multiput(0,900)(7200,0){2}{\line(0,1){450}}
\put(0,1350){\line(1,0){7200}}
\multiput(7200,-900)(7200,0){2}{\line(0,-1){450}}
\put(7200,-1350){\line(1,0){7200}}
\put(0,600){\usebox{\toe}}
\put(7200,600){\usebox{\tow}}
\put(0,-900){\line(0,-1){900}}
\put(0,-1800){\line(1,0){15000}}
\put(15000,-1800){\line(0,1){2400}}
\put(15000,600){\line(-1,0){300}}
\end{picture}
\]
and is obtained by parabolic induction with $S'=\{\alpha_1,\ldots,\alpha_{p+2}, \alpha_{p+q+3}\}$ and $\s'$ equal to the primitive system
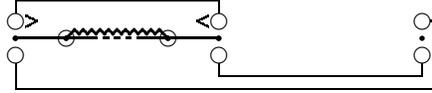
\[
\begin{picture}(16000,3000)(-300,-1500)
\put(0,0){\usebox{\aone}}
\put(0,0){\usebox{\edge}}
\put(1800,0){\usebox{\shortam}}
\put(5400,0){\usebox{\edge}}
\put(7200,0){\usebox{\aone}}
\put(14400,0){\usebox{\aone}}
\multiput(0,900)(7200,0){2}{\line(0,1){450}}
\put(0,1350){\line(1,0){7200}}
\multiput(7200,-900)(7200,0){2}{\line(0,-1){450}}
\put(7200,-1350){\line(1,0){7200}}
\put(0,600){\usebox{\toe}}
\put(7200,600){\usebox{\tow}}
\put(0,-900){\line(0,-1){900}}
\put(0,-1800){\line(1,0){15000}}
\put(15000,-1800){\line(0,1){2400}}
\put(15000,600){\line(-1,0){300}}
\end{picture}
\]
The (essential) quotient $\widehat \s/\Delta'$ is
\[
\begin{picture}(16000,3000)(-300,-1500)
\put(0,0){\usebox{\edge}}
\put(1800,0){\usebox{\susp}}
\put(5400,0){\usebox{\edge}}
\put(7200,0){\usebox{\edge}}
\put(9000,0){\usebox{\susp}}
\put(9000,0){\usebox{\wcircle}}
\put(12600,0){\usebox{\wcircle}}
\put(12600,0){\usebox{\leftbiedge}}
\put(5400,0){\usebox{\wcircle}}
\put(7200,0){\usebox{\wcircle}}
\put(14400,0){\usebox{\wcircle}}
\multiput(7200,-300)(7200,0){2}{\line(0,-1){450}}
\put(7200,-750){\line(1,0){7200}}
\end{picture}
\]
The commutator group $\widehat L'$ of $\widehat L$ is isomorphic to $\SL(p+1)\times\SL(2)\times\SL(q-1)$, and $W$ is $\widehat L'$-isomorphic to $\C^{p+1}\otimes\C^2$. Then $H^u$ is obtained from $K^u$ removing the summand $\C^{p+1}\otimes\C^2$ from $\Lie K^u$; one checks directly that $H=LH^u$ is a spherical subgroup of $G$, equal to its normalizer. We denote by $X$ the wonderful variety with general isotropy $H$, and we must show that $\s_X=\s$.

We must consider all other minimal quotients of $\s$; other than $\s_1=\s/\Delta'$ we have the following two other quotients:
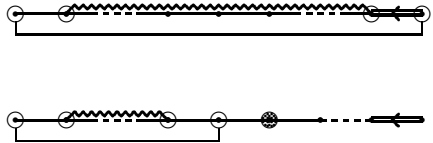
\[
\begin{picture}(16000,3000)(-300,-1500)
\put(0,0){\usebox{\edge}}
\put(1800,0){\usebox{\susp}}
\put(5400,0){\usebox{\edge}}
\put(7200,0){\usebox{\edge}}
\put(9000,0){\usebox{\susp}}
\multiput(1800,0)(10800,0){2}{\circle{600}}\multiput(1800,0)(25,25){13}{\circle*{70}}\put(2100,300){\multiput(0,0)(300,0){34}{\multiput(0,0)(25,-25){7}{\circle*{70}}}\multiput(150,-150)(300,0){34}{\multiput(0,0)(25,25){7}{\circle*{70}}}}\multiput(12600,0)(-25,25){13}{\circle*{70}}
\put(12600,0){\usebox{\leftbiedge}}
\put(0,0){\usebox{\wcircle}}
\put(14400,0){\usebox{\wcircle}}
\multiput(0,-300)(14400,0){2}{\line(0,-1){450}}
\put(0,-750){\line(1,0){14400}}
\end{picture}
\]
\[
\begin{picture}(16000,3000)(-300,-1500)
\put(0,0){\usebox{\edge}}
\put(1800,0){\usebox{\susp}}
\put(5400,0){\usebox{\edge}}
\put(1800,0){\usebox{\shortam}}
\put(7200,0){\multiput(0,0)(1800,0){2}{\usebox{\edge}}\put(3600,0){\usebox{\shortsusp}}\put(5400,0){\usebox{\leftbiedge}}\put(1800,0){\usebox{\gcircle}} }
\put(0,0){\usebox{\wcircle}}
\put(7200,0){\usebox{\wcircle}}
\multiput(0,-300)(7200,0){2}{\line(0,-1){450}}
\put(0,-750){\line(1,0){7200}}
\end{picture}
\]
The former (denote it by $\s_2$) has rank $2$, and according to \cite{W} the general isotropy $H_2$ is equal to the product of $H^u$ with the Levi subgroup of type $\mathsf A_{p+q+1}\times \mathsf A_1$ containing $T$. The latter quotient (denote it by $\s_3$) is obtained adding to a spherical system of rank $2$ a tail of type $c_{q+2}$. According to \cite[\S6]{BP11} the general isotropy $H_3$ is equal to $LH_3^u$, where $\Lie H_3^u$ is obtained from $\Lie Q^u$ by removing the summand $\C^{p+1}\otimes\C^2$.

The inclusions $H\subset H_i$ for $i\in\{1,2,3\}$ are clear, hence $\s_X$ admits $\s_1$, $\s_2$ and $\s_3$ as quotients; now we must show that $\s_X=\s$.

The quotients $\s_1$, $\s_2$ and $\s_3$  of $\s_X$ have one color each moved by exactly two simple roots, and each of these simple roots (namely $\alpha_1$, $\alpha_{p+2}$ and $\alpha_{p+q+3}$) occurs exactly twice; we deduce that they are spherical roots of $X$ and that $\A_X$ has (possibly among others) three colors in the following configuration:
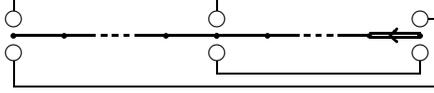
\[
\begin{picture}(16000,3000)(-300,-1500)
\put(0,0){\usebox{\aone}}
\put(0,0){\usebox{\edge}}
\put(1800,0){\usebox{\susp}}
\put(5400,0){\usebox{\edge}}
\put(7200,0){\usebox{\edge}}
\put(9000,0){\usebox{\susp}}
\put(12600,0){\usebox{\leftbiedge}}
\put(7200,0){\usebox{\aone}}
\put(14400,0){\usebox{\aone}}
\multiput(0,900)(7200,0){2}{\line(0,1){450}}
\put(0,1350){\line(1,0){7200}}
\multiput(7200,-900)(7200,0){2}{\line(0,-1){450}}
\put(7200,-1350){\line(1,0){7200}}
\put(0,-900){\line(0,-1){900}}
\put(0,-1800){\line(1,0){15000}}
\put(15000,-1800){\line(0,1){2400}}
\put(15000,600){\line(-1,0){300}}
\end{picture}
\]
The quotient $\s_2$ of $\s_X$ implies that all simple roots are in $\supp\Sigma_X$, and property (S) of Definition~\ref{def:system} excludes that $\alpha_1$, $\alpha_{p+2}$ and $\alpha_{p+q+3}$ are in the support of any other spherical root. Repeated application of property ($\Sigma 2$) of Definition~\ref{def:system} (with $\sigma$ equal to $\alpha_1$, $\alpha_{p+2}$ and $\alpha_{p+q+3}$) shows that no spherical root of $X$ is a sum of two orthogonal simple roots. 

At this point properties (A1) and (A2) of Definition~\ref{def:system}, applied to the spherical roots supported on $\alpha_2$, $\alpha_{p+1}$, $\alpha_{p+3}$ and $\alpha_{p+q+2}$, imply that $\{\alpha_2,\ldots,\alpha_{p+1}\}$ is the support of a single spherical root, as well as $\{\alpha_{p+3},\ldots,\alpha_{p+q+2}\}$. Property (S) of Definition~\ref{def:system} yields that $\alpha_2+\ldots+\alpha_{p+1}$ and $\alpha_{p+3}+\ldots+\alpha_{p+q+2}$ are spherical roots of $X$ and that the Cartan pairing of $\s_X$ is that of $\s$. It follows $\s_X = \s$.

The above techniques, with repeated application of the statements of Definition~\ref{def:system}, are rather standard and are used in a similar way to complete the analysis of the remaining $11$ cases. They will also be used implicitly in the rest of the paper.

\subsection{Remaining cases} We are left with 12 cases, which we subdivide as follows:
\begin{enumerate}
\item $\mathsf{ab}^\mathsf{y}(p-1,p)$, $\mathsf{ag}^\mathsf{y}(1,2)$, $\mathsf{b}^\mathsf{w}(3)$, $\mathsf S$-70, $\mathsf S$-72, $\mathsf S$-96 and $\mathsf T$-28;
\item $\mathsf{b}^\mathsf{y}(4)$;
\item $\mathsf S$-63, $\mathsf S$-81 and $\mathsf S$-85;
\item $\mathsf T$-21.
\end{enumerate}

\subsubsection{}

We start with the case $\mathsf{ab}^\mathsf{y}(p-1,p)$. Here $G$ is of type $\mathsf A_{p-1}\times\mathsf B_p$, with $p\geq2$: $S=\{\alpha_1,\ldots,\alpha_{p-1},\alpha'_1,\ldots,\alpha'_p\}$. Let us consider the quotient $\s\to\s/\Delta'$ of type $(\mathscr L)$ of constant defect described by the following diagrams.
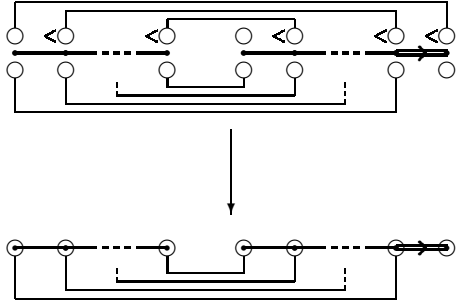
\[\begin{picture}(15900,10500)
\put(300,8700){
\multiput(0,0)(8100,0){2}{
\put(0,0){\usebox{\edge}}
\put(1800,0){\usebox{\susp}}
\multiput(0,0)(1800,0){2}{\usebox{\aone}}
\put(5400,0){\usebox{\aone}}
\put(5400,600){\usebox{\tow}}
\put(1800,600){\usebox{\tow}}
}
\put(15300,0){\usebox{\aone}}
\put(15300,600){\usebox{\tow}}
\put(13500,0){\usebox{\rightbiedge}}
\multiput(0,900)(15300,0){2}{\line(0,1){900}}
\put(0,1800){\line(1,0){15300}}
\multiput(1800,900)(11700,0){2}{\line(0,1){600}}
\put(1800,1500){\line(1,0){11700}}
\multiput(5400,900)(4500,0){2}{\line(0,1){300}}
\put(5400,1200){\line(1,0){4500}}
\multiput(0,-900)(13500,0){2}{\line(0,-1){1200}}
\put(0,-2100){\line(1,0){13500}}
\put(1800,-900){\line(0,-1){900}}
\put(1800,-1800){\line(1,0){9900}}
\multiput(11700,-1800)(0,300){3}{\line(0,1){150}}
\multiput(3600,-1500)(0,300){2}{\line(0,1){150}}
\put(3600,-1500){\line(1,0){6300}}
\put(9900,-1500){\line(0,1){600}}
\multiput(5400,-1200)(2700,0){2}{\line(0,1){300}}
\put(5400,-1200){\line(1,0){2700}}
}
\put(7950,6000){\vector(0,-1){3000}}
\put(300,1800){
\multiput(0,0)(8100,0){2}{
\put(0,0){\usebox{\edge}}
\put(1800,0){\usebox{\susp}}
\multiput(0,0)(1800,0){2}{\usebox{\wcircle}}
\put(5400,0){\usebox{\wcircle}}
}
\put(15300,0){\usebox{\wcircle}}
\put(13500,0){\usebox{\rightbiedge}}
\multiput(0,-300)(13500,0){2}{\line(0,-1){1500}}
\put(0,-1800){\line(1,0){13500}}
\put(1800,-300){\line(0,-1){1200}}
\put(1800,-1500){\line(1,0){9900}}
\multiput(11700,-1500)(0,300){3}{\line(0,1){150}}
\multiput(3600,-1200)(0,300){2}{\line(0,1){150}}
\put(3600,-1200){\line(1,0){6300}}
\put(9900,-1200){\line(0,1){900}}
\multiput(5400,-900)(2700,0){2}{\line(0,1){600}}
\put(5400,-900){\line(1,0){2700}}
}
\end{picture}\]
The spherical system $\s/\Delta'$ is geometrically realizable, it is parabolic induction of the spherical system of a wonderful $Q/Q^r$-variety with affine open $Q/Q^r$-orbit, where $Q$ is the maximal parabolic subgroup of $G$ containing $B_-$ corresponding to $S\smallsetminus\{\alpha'_p\}$. Indeed, the group $Q/Q^r$ is semi-simple of type $\mathsf A_{p-1}\times\mathsf A_{p-1}$ and we can choose the principal isotropy group $K$ corresponding to $\s/\Delta'$ as the subgroup containing $Q^r$ and such that $K/Q^r$ is the semi-simple subgroup of type $\mathsf A_{p-1}$ diagonally embedded in $Q/Q^r$, a very reductive subgroup. Set $Q=Q^u\,L_Q$ with $L_Q\supset T$ and $K=K^u\,L_K$ with $L_K\subset L_Q$. As $L_Q$-modules we have
\[\Lie Q^u\cong V(-\alpha'_p)\oplus[\Lie Q^u,\Lie Q^u],\]
where $V(-\alpha'_p)$ is the simple $L_Q$-module of highest $T$-weight $-\alpha'_p$. This module remains simple under the action of $L_K$. Let us now choose the principal isotropy group $H\subset K$ corresponding to $\s$: we can take $H=H^u\,L$ with $H^u=(Q^u,Q^u)$ and $L=L_K$.

The subgroup $H$ is spherical and self-normalizing. To prove that it is the principal isotropy group of the wonderful variety with spherical system $\s$ it is enough to notice that there is no other cuspidal spherical system admitting $\s/\Delta'$ as quotient.

Unless otherwise stated, the same argument shows that the subgroups $H$ we give for all the remaining cases correspond to the expected spherical systems.

The other cases of this block are very similar (with one slight exception). We will put them one after the other keeping the same notation.

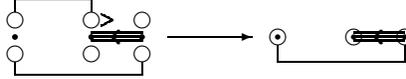
\[\begin{picture}(14400,2700)
\put(300,1350){
\put(0,0){\usebox{\vertex}}
\put(2700,0){\usebox{\lefttriedge}}
\multiput(0,0)(4500,0){2}{\usebox{\aone}}
\put(2700,0){\usebox{\aone}}
\put(2700,600){\usebox{\toe}}
\multiput(0,1350)(2700,0){2}{\line(0,-1){450}}
\put(0,1350){\line(1,0){2700}}
\multiput(0,-1350)(4500,0){2}{\line(0,1){450}}
\put(0,-1350){\line(1,0){4500}}
}
\put(5700,1350){\vector(1,0){3000}}
\put(9600,1350){
\put(0,0){\usebox{\vertex}}
\put(2700,0){\usebox{\lefttriedge}}
\multiput(0,0)(4500,0){2}{\usebox{\wcircle}}
\put(2700,0){\usebox{\wcircle}}
\multiput(0,-900)(4500,0){2}{\line(0,1){600}}
\put(0,-900){\line(1,0){4500}}
}
\end{picture}\]
In this case $G$ is of type $\mathsf A_1\times\mathsf G_2$, $S=\{\alpha_1,\alpha'_1,\alpha'_2\}$. The parabolic subgroup $Q\supset B_-$ corresponds to $S\smallsetminus\{\alpha'_1\}$ and has semi-simple type $\mathsf A_1\times\mathsf A_1$. The group $K/Q^r$ is the semi-simple subgroup of type $\mathsf A_1$ diagonally embedded in $Q/Q^r$. As above we can take $H=H^u\,L$ with $H^u=(Q^u,Q^u)$ and $L=L_K$.

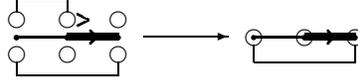
\[\begin{picture}(12600,2700)
\put(300,1350){
\put(0,0){\usebox{\dynkinbthree}}
\multiput(0,0)(1800,0){3}{\usebox{\aone}}
\put(1800,600){\usebox{\toe}}
\multiput(0,1350)(1800,0){2}{\line(0,-1){450}}
\put(0,1350){\line(1,0){1800}}
\multiput(0,-1350)(3600,0){2}{\line(0,1){450}}
\put(0,-1350){\line(1,0){3600}}
}
\put(4800,1350){\vector(1,0){3000}}
\put(8700,1350){
\put(0,0){\usebox{\dynkinbthree}}
\multiput(0,0)(1800,0){3}{\usebox{\wcircle}}
\multiput(0,-900)(3600,0){2}{\line(0,1){600}}
\put(0,-900){\line(1,0){3600}}
}
\end{picture}\]
In this case $G$ is of type $\mathsf B_3$, $S=\{\alpha_1,\alpha_2,\alpha_3\}$. The parabolic subgroup $Q\supset B_-$ corresponds to $S\smallsetminus\{\alpha_2\}$ and has semi-simple type $\mathsf A_1\times\mathsf A_1$. The group $K/Q^r$ is the semi-simple subgroup of type $\mathsf A_1$ diagonally embedded in $Q/Q^r$. Here the simple $L_Q$-module $V_{L_Q}(-\alpha_2)$ does not remain simple under the action of $L_K$: as $L_K$-modules 
\[V_{L_Q}(-\alpha_2)\cong V_{L_K}(-\alpha_2)\oplus W,\]
where $W$ is simple of dimension 2. We take $H=H^u\,L$ with $L=L_K$ and $\Lie H^u$ equal to the $L$-complementary of $W$ in $\Lie Q^u$.

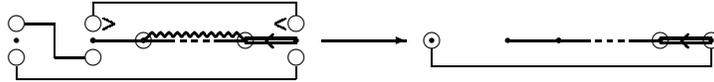
\[\begin{picture}(25200,2700)
\put(300,1350){
\put(4500,0){\usebox{\shortam}}
\put(2700,0){\usebox{\edge}}
\put(8100,0){\usebox{\leftbiedge}}
\put(0,0){\usebox{\aone}}
\multiput(2700,0)(7200,0){2}{\usebox{\aone}}
\multiput(2700,1350)(7200,0){2}{\line(0,-1){450}}
\put(2700,1350){\line(1,0){7200}}
\multiput(0,-1350)(9900,0){2}{\line(0,1){450}}
\put(0,-1350){\line(1,0){9900}}
\multiput(300,600)(1050,-1200){2}{\line(1,0){1050}}
\put(1350,600){\line(0,-1){1200}}
\put(2700,600){\usebox{\toe}}
\put(9900,600){\usebox{\tow}}
}
\put(11100,1350){\vector(1,0){3000}}
\put(15000,1350){
\put(4500,0){\usebox{\susp}}
\put(2700,0){\usebox{\edge}}
\put(8100,0){\usebox{\leftbiedge}}
\put(0,0){\usebox{\vertex}}
\multiput(0,0)(9900,0){2}{\usebox{\wcircle}}
\put(8100,0){\usebox{\wcircle}}
\multiput(0,-900)(9900,0){2}{\line(0,1){600}}
\put(0,-900){\line(1,0){9900}}
}
\end{picture}\]
In this case $G$ is of type $\mathsf A_1\times\mathsf C_{p+2}$ with $p\geq1$, $S=\{\alpha_1,\alpha'_1,\ldots,\alpha'_{p+2}\}$. The parabolic subgroup $Q\supset B_-$ corresponds to $S\smallsetminus\{\alpha'_{p+1}\}$ and has semi-simple type $\mathsf A_1\times\mathsf A_p\times\mathsf A_1$. The group $K/Q^r$ is the semi-simple subgroup of type $\mathsf A_1\times\mathsf A_p$ with the first factor diagonally embedded in the first and third factor of $Q/Q^r$. As in the first and second case of this block we can take $H=H^u\,L$ with $H^u=(Q^u,Q^u)$ and $L=L_K$.

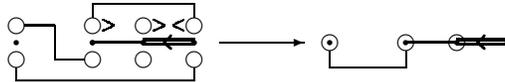
\[\begin{picture}(18000,2700)
\put(300,1350){
\put(2700,0){\usebox{\edge}}
\put(4500,0){\usebox{\leftbiedge}}
\put(0,0){\usebox{\aone}}
\multiput(2700,0)(1800,0){3}{\usebox{\aone}}
\multiput(2700,1350)(3600,0){2}{\line(0,-1){450}}
\put(2700,1350){\line(1,0){3600}}
\multiput(0,-1350)(6300,0){2}{\line(0,1){450}}
\put(0,-1350){\line(1,0){6300}}
\multiput(300,600)(1050,-1200){2}{\line(1,0){1050}}
\put(1350,600){\line(0,-1){1200}}
\multiput(2700,600)(1800,0){2}{\usebox{\toe}}
\put(6300,600){\usebox{\tow}}
}
\put(7500,1350){\vector(1,0){3000}}
\put(11400,1350){
\put(2700,0){\usebox{\edge}}
\put(4500,0){\usebox{\leftbiedge}}
\put(0,0){\usebox{\vertex}}
\put(0,0){\usebox{\wcircle}}
\multiput(2700,0)(1800,0){2}{\usebox{\wcircle}}
\multiput(0,-900)(2700,0){2}{\line(0,1){600}}
\put(0,-900){\line(1,0){2700}}
}
\end{picture}\]
In this case $G$ is of type $\mathsf A_1\times\mathsf C_3$, $S=\{\alpha_1,\alpha'_1,\alpha'_2,\alpha'_3\}$. The parabolic subgroup $Q\supset B_-$ corresponds to $S\smallsetminus\{\alpha'_2\}$ and has semi-simple type $\mathsf A_1\times\mathsf A_1\times\mathsf A_1$. The group $K/Q^r$ is the semi-simple subgroup of type $\mathsf A_1\times\mathsf A_1$ with the first factor diagonally embedded in the first and second factor of $Q/Q^r$. As in the above case we can take $H=H^u\,L$ with $H^u=(Q^u,Q^u)$ and $L=L_K$.

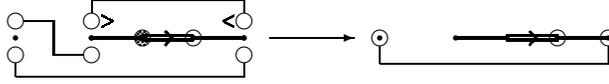
\[\begin{picture}(21600,2700)
\put(300,1350){
\put(0,0){\usebox{\aone}}
\put(2700,0){\usebox{\dynkinf}}
\multiput(2700,0)(5400,0){2}{\usebox{\aone}}
\put(4500,0){\usebox{\bsecondtwo}}
\multiput(2700,1350)(5400,0){2}{\line(0,-1){450}}
\put(2700,1350){\line(1,0){5400}}
\multiput(0,-1350)(8100,0){2}{\line(0,1){450}}
\put(0,-1350){\line(1,0){8100}}
\multiput(300,600)(1050,-1200){2}{\line(1,0){1050}}
\put(1350,600){\line(0,-1){1200}}
\put(2700,600){\usebox{\toe}}
\put(8100,600){\usebox{\tow}}
}
\put(9300,1350){\vector(1,0){3000}}
\put(13200,1350){
\put(0,0){\usebox{\vertex}}
\put(2700,0){\usebox{\dynkinf}}
\multiput(0,0)(8100,0){2}{\usebox{\wcircle}}
\put(6300,0){\usebox{\wcircle}}
\multiput(0,-900)(8100,0){2}{\line(0,1){600}}
\put(0,-900){\line(1,0){8100}}
}
\end{picture}\]
In this case $G$ is of type $\mathsf A_1\times\mathsf F_4$, $S=\{\alpha_1,\alpha'_1,\alpha'_2,\alpha'_3,\alpha'_4\}$. The parabolic subgroup $Q\supset B_-$ corresponds to $S\smallsetminus\{\alpha'_3\}$ and has semi-simple type $\mathsf A_1\times\mathsf A_2\times\mathsf A_1$. The group $K/Q^r$ is the semi-simple subgroup of type $\mathsf A_1\times\mathsf A_2$ with the first factor diagonally embedded in the first and third factor of $Q/Q^r$. As in the above case we can take $H=H^u\,L$ with $H^u=(Q^u,Q^u)$ and $L=L_K$.

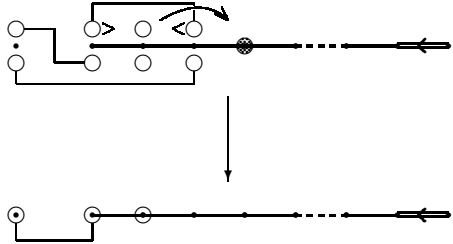
\[\begin{picture}(15600,8400)
\put(300,6900){
\put(2700,0){\usebox{\edge}}
\put(4500,0){\usebox{\edge}}
\put(0,0){\usebox{\aone}}
\multiput(2700,0)(1800,0){3}{\usebox{\aone}}
\put(2700,1500){\line(0,-1){600}}
\put(2700,1500){\line(1,0){3600}}
\put(6300,900){\line(0,1){350}}
\put(6300,1425){\line(0,1){75}}
\multiput(0,-1350)(6300,0){2}{\line(0,1){450}}
\put(0,-1350){\line(1,0){6300}}
\multiput(300,600)(1050,-1200){2}{\line(1,0){1050}}
\put(1350,600){\line(0,-1){1200}}
\put(2700,600){\usebox{\toe}}
\put(6300,600){\usebox{\tow}}
\put(6300,0){\usebox{\pluscsecondm}}
\put(4500,600){\usebox{\tobe}}
}
\put(7800,5100){\vector(0,-1){3000}}
\put(300,900){
\put(0,0){\usebox{\vertex}}
\multiput(2700,0)(1800,0){4}{\usebox{\edge}}
\put(9900,0){\usebox{\shortsusp}}
\put(11700,0){\usebox{\dynkincthree}}
\multiput(0,0)(2700,0){2}{\usebox{\wcircle}}
\put(4500,0){\usebox{\wcircle}}
\multiput(0,-900)(2700,0){2}{\line(0,1){600}}
\put(0,-900){\line(1,0){2700}}
}
\end{picture}\]
This can be seen as a generalization of the fifth case, $G$ is of type $\mathsf A_1\times\mathsf C_{p+3}$ with $p\geq1$, $S=\{\alpha_1,\alpha'_1,\ldots,\alpha'_{p+3}\}$. The parabolic subgroup $Q\supset B_-$ corresponds to $S\smallsetminus\{\alpha'_2\}$ and has semi-simple type $\mathsf A_1\times\mathsf A_1\times\mathsf A_{p+1}$. The group $K/Q^r$ is the semi-simple subgroup of type $\mathsf A_1\times\mathsf A_{p+1}$ with the first factor diagonally embedded in the first and second factor of $Q/Q^r$. As in the above case we can take $H=H^u\,L$ with $H^u=(Q^u,Q^u)$ and $L=L_K$.

\subsubsection{}

We discuss the case $\mathsf{b}^\mathsf{y}(4)$. As above, we consider the quotient:

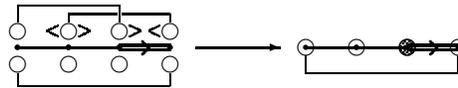
\[\begin{picture}(16200,2850)(-300,-1350)
\put(0,0){\usebox{\dynkinbfour}}
\multiput(0,0)(1800,0){4}{\usebox{\aone}}
\multiput(0,1500)(3600,0){2}{\line(0,-1){600}}
\put(0,1500){\line(1,0){3600}}
\multiput(1800,1200)(3600,0){2}{\line(0,-1){300}}
\multiput(1800,1200)(1900,0){2}{\line(1,0){1700}}
\multiput(0,-1350)(5400,0){2}{\line(0,1){450}}
\put(0,-1350){\line(1,0){5400}}
\multiput(1800,600)(3600,0){2}{\usebox{\tow}}
\multiput(1800,600)(1800,0){2}{\usebox{\toe}}
\put(6300,0){\vector(1,0){3000}}
\put(10200,0){
\put(0,0){\usebox{\dynkinbfour}}\multiput(0,0)(1800,0){4}{\usebox{\wcircle}}\multiput(0,-900)(5400,0){2}{\line(0,1){600}}\put(0,-900){\line(1,0){5400}}\put(3600,0){\usebox{\gcircle}}
}
\end{picture}\]
The group $G$ is of type $\mathsf B_4$, $S=\{\alpha_1,\alpha_2,\alpha_3,\alpha_4\}$. The parabolic subgroup $Q\supset B_-$ corresponds to $S\smallsetminus\{\alpha_2\}$ and has semi-simple type $\mathsf A_1\times\mathsf B_2$. The group $K/Q^r$ is the semi-simple subgroup of type $\mathsf A_1\times\mathsf A_1$ with the first factor diagonally embedded in the first and third factor of a semi-simple subgroup $K_2/Q^r$ of $Q/Q^r$ of type $\mathsf A_1\times\mathsf A_1\times\mathsf A_1$. The simple $L_Q$-module $V_{L_Q}(-\alpha_2)$ does not remain simple under the action of $L_K$: as $L_{K_2}$-modules 
\[V_{L_Q}(-\alpha_2)\cong V_{L_{K_2}}(-\alpha_2)\oplus W_2,\]
where $W_2$ is simple of dimension 2, and as $L_K$-modules
\[V_{L_{K_2}}(-\alpha_2)\cong V_{L_K}(-\alpha_2)\oplus W,\]
where $W$ is simple of dimension 2. We take $H=H^u\,L$ with $L=L_K$ and $\Lie H^u$ equal to the $L$-complementary of $W$ in $\Lie Q^u$.

We now prove that the subgroup $H$ corresponds to $\s$. Notice that we could have taken as $\Lie H^u$ the $L$-complementry of $W_2$ in $\Lie Q^u$ obtaining a self-normalizing spherical subgroup, too. Indeed, let us also consider the following quotient of type $(\mathscr L)$ of constant defect (with fiber of dimension 2).
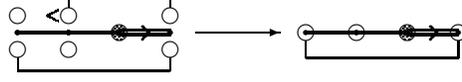
\[\begin{picture}(16200,2700)
\put(-1800,0){
\multiput(2100,1350)(1800,0){2}{\usebox{\edge}}
\multiput(2100,1350)(1800,0){2}{\usebox{\aone}}
\put(7500,1350){\usebox{\aone}}
\put(5700,1350){\usebox{\btwo}}
\multiput(3900,2250)(3600,0){2}{\line(0,1){450}}\put(3900,2700){\line(1,0){3600}}\multiput(2100,0)(5400,0){2}{\line(0,1){450}}\put(2100,0){\line(1,0){5400}}
\put(3900,1950){\usebox{\tow}}
}
\put(6600,1350){\vector(1,0){3000}}
\put(8400,0){
\multiput(2100,1350)(1800,0){2}{\usebox{\edge}}
\multiput(2100,1350)(1800,0){2}{\usebox{\wcircle}}
\put(7500,1350){\usebox{\wcircle}}
\put(5700,1350){\usebox{\btwo}}
\multiput(2100,450)(5400,0){2}{\line(0,1){600}}\put(2100,450){\line(1,0){5400}}
}
\end{picture}\]
This is the (only) other possible choice of a spherical system corresponding to $H$. To show that $H$ actually corresponds to the spherical system represented by the first diagram above, it is enough to notice that the second one admits also the following (non-minimal) quotient, which would correspond to the inclusion of $H$ into a semi-simple subgroup of type $D_4$.
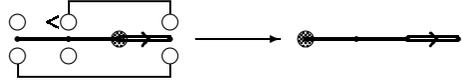
\[\begin{picture}(16200,2700)
\put(-1800,0){
\multiput(2100,1350)(1800,0){2}{\usebox{\edge}}
\multiput(2100,1350)(1800,0){2}{\usebox{\aone}}
\put(7500,1350){\usebox{\aone}}
\put(5700,1350){\usebox{\btwo}}
\multiput(3900,2250)(3600,0){2}{\line(0,1){450}}\put(3900,2700){\line(1,0){3600}}\multiput(2100,0)(5400,0){2}{\line(0,1){450}}\put(2100,0){\line(1,0){5400}}
\put(3900,1950){\usebox{\tow}}
}
\put(6600,1350){\vector(1,0){3000}}
\put(8400,0){
\multiput(2100,1350)(1800,0){2}{\usebox{\edge}}
\put(5700,1350){\usebox{\rightbiedge}}
\put(2100,1350){\usebox{\gcircle}}
}
\end{picture}\]

\subsubsection{}

We start with $\mathsf S$-63.

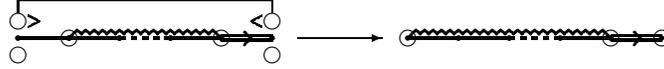
\[\begin{picture}(23400,2250)
\put(300,900){
\put(0,0){\usebox{\edge}}
\put(7200,0){\usebox{\rightbiedge}}
\multiput(0,0)(9000,0){2}{\usebox{\aone}}
\put(1800,0){\usebox{\mediumam}}
\put(0,600){\usebox{\toe}}
\put(9000,600){\usebox{\tow}}
\multiput(0,1350)(9000,0){2}{\line(0,-1){450}}
\put(0,1350){\line(1,0){9000}}
}
\put(10200,900){\vector(1,0){3000}}
\put(14100,900){
\multiput(0,0)(5400,0){2}{\usebox{\edge}}
\put(1800,0){\usebox{\edge}}
\put(7200,0){\usebox{\rightbiedge}}
\put(3600,0){\usebox{\shortsusp}}
\multiput(0,0)(9000,0){2}{\usebox{\wcircle}}
\put(7200,0){\usebox{\wcircle}}
\put(0,0){\multiput(0,0)(25,25){13}{\circle*{70}}\multiput(300,300)(300,0){22}{\multiput(0,0)(25,-25){7}{\circle*{70}}}\multiput(450,150)(300,0){22}{\multiput(0,0)(25,25){7}{\circle*{70}}}\multiput(7200,0)(-25,25){13}{\circle*{70}}}
}
\end{picture}\]
The group $G$ is of type $\mathsf B_{q+2}$ with $q\geq1$, $S=\{\alpha_1,\ldots,\alpha_{q+2}\}$. The parabolic subgroup $Q\supset B_-$ corresponds to $S\smallsetminus\{\alpha_{q+2}\}$ and has semi-simple type $\mathsf A_{q+1}$. The subgroup $K/Q^r$ of $Q/Q^r$ is reductive of semi-simple type $\mathsf A_q$. The simple $L_Q$-module $V_{L_Q}(-\alpha_{q+2})$ does not remain simple under the action of $L_K$: as $L_K$-modules 
\[V_{L_Q}(-\alpha_{q+2})\cong \C\oplus W,\]
where $W$ is simple of dimension $q+1$. We take $H=H^u\,L$ with $L=L_K$ and $\Lie H^u$ equal to the $L$-complementary of $W$ in $\Lie Q^u$.

This generalizes to the case  $\mathsf S$-81:
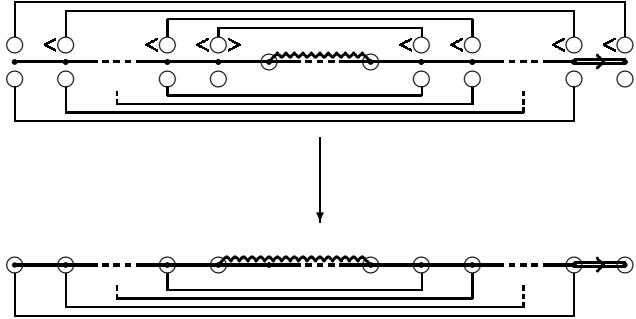
\[\begin{picture}(22200,11100)
\put(300,9000){
\multiput(0,0)(14400,0){2}{
\put(1800,0){\usebox{\susp}}
\multiput(0,0)(5400,0){2}{\multiput(0,0)(1800,0){2}{\usebox{\aone}}}
\multiput(5400,600)(1800,0){2}{\usebox{\tow}}
\put(1800,600){\usebox{\tow}}
}
\multiput(0,0)(5400,0){2}{\usebox{\edge}}
\put(14400,0){\usebox{\edge}}
\put(19800,0){\usebox{\rightbiedge}}
\multiput(7200,0)(5400,0){2}{\usebox{\edge}}
\put(9000,0){\usebox{\shortam}}
\put(7200,600){\usebox{\toe}}
\put(14400,600){\usebox{\tow}}
\multiput(0,900)(21600,0){2}{\line(0,1){1200}}
\put(0,2100){\line(1,0){21600}}
\multiput(1800,900)(18000,0){2}{\line(0,1){900}}
\put(1800,1800){\line(1,0){18000}}
\multiput(5400,900)(10800,0){2}{\line(0,1){600}}
\put(5400,1500){\line(1,0){10800}}
\multiput(7200,900)(7200,0){2}{\line(0,1){300}}
\put(7200,1200){\line(1,0){7200}}
\multiput(0,-900)(19800,0){2}{\line(0,-1){1200}}
\put(0,-2100){\line(1,0){19800}}
\put(1800,-900){\line(0,-1){900}}
\put(1800,-1800){\line(1,0){16200}}
\multiput(18000,-1800)(0,300){3}{\line(0,1){150}}
\multiput(3600,-1500)(0,300){2}{\line(0,1){150}}
\put(3600,-1500){\line(1,0){12600}}
\put(16200,-1500){\line(0,1){600}}
\multiput(5400,-1200)(9000,0){2}{\line(0,1){300}}
\put(5400,-1200){\line(1,0){9000}}
}
\put(11100,6300){\vector(0,-1){3000}}
\put(300,1800){
\multiput(0,0)(14400,0){2}{
\put(1800,0){\usebox{\susp}}
\multiput(0,0)(5400,0){2}{\multiput(0,0)(1800,0){2}{\usebox{\wcircle}}}
}
\multiput(0,0)(5400,0){2}{\usebox{\edge}}
\multiput(7200,0)(5400,0){2}{\usebox{\edge}}
\put(14400,0){\usebox{\edge}}
\put(19800,0){\usebox{\rightbiedge}}
\put(9000,0){\usebox{\susp}}
\put(0,300){
\multiput(0,-600)(19800,0){2}{\line(0,-1){1500}}
\put(0,-2100){\line(1,0){19800}}
\put(1800,-600){\line(0,-1){1200}}
\put(1800,-1800){\line(1,0){16200}}
\multiput(18000,-1800)(0,300){3}{\line(0,1){150}}
\multiput(3600,-1500)(0,300){2}{\line(0,1){150}}
\put(3600,-1500){\line(1,0){12600}}
\put(16200,-1500){\line(0,1){900}}
\multiput(5400,-1200)(9000,0){2}{\line(0,1){600}}
\put(5400,-1200){\line(1,0){9000}}
}
\put(12600,0){\usebox{\wcircle}}
\put(7200,0){\multiput(0,0)(25,25){13}{\circle*{70}}\multiput(300,300)(300,0){16}{\multiput(0,0)(25,-25){7}{\circle*{70}}}\multiput(450,150)(300,0){16}{\multiput(0,0)(25,25){7}{\circle*{70}}}\multiput(5400,0)(-25,25){13}{\circle*{70}}}
}
\end{picture}\]
In this case the group $G$ is of type $\mathsf B_{2p+q+2}$ with $p,q\geq1$, $S=\{\alpha_1,\ldots,\alpha_{2p+q+2}\}$. The parabolic subgroup $Q\supset B_-$ corresponds to $S\smallsetminus\{\alpha_{2p+q+2}\}$ and has semi-simple type $\mathsf A_{2p+q+1}$. The subgroup $K/Q^r$ of $Q/Q^r$ is reductive of semi-simple type $\mathsf A_p\times\mathsf A_{p+q}$. The simple $L_Q$-module $V_{L_Q}(-\alpha_{2p+q+2})$ does not remain simple under the action of $L_K$: as $L_K$-modules 
\[V_{L_Q}(-\alpha_{2p+q+2})\cong W_1\oplus W_2,\]
where $W_1$ and $W_2$ are simple of dimension $p+1$ and $p+q+1$, respectively. We take $H=H^u\,L$ with $L=L_K$ and $\Lie H^u$ equal to the $L$-complementary of $W_2$ in $\Lie Q^u$.

Let us now consider the case $\mathsf S$-85.
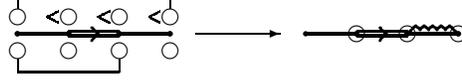
\[\begin{picture}(16200,2700)(-300,-1350)
\put(0,0){\usebox{\dynkinffour}}
\multiput(0,0)(1800,0){4}{\usebox{\aone}}
\multiput(0,1350)(5400,0){2}{\line(0,-1){450}}
\put(0,1350){\line(1,0){5400}}
\multiput(0,-1350)(3600,0){2}{\line(0,1){450}}
\put(0,-1350){\line(1,0){3600}}
\multiput(1800,600)(1800,0){3}{\usebox{\tow}}
\put(6300,0){\vector(1,0){3000}}
\put(10200,0){
\put(0,0){\usebox{\dynkinffour}}\put(1800,0){\usebox{\wcircle}}\put(3600,0){\usebox{\atwo}}
}
\end{picture}\]
In this case the parabolic subgroup $Q\supset B_-$ corresponds to $S\smallsetminus\{\alpha_2\}$ and has semi-simple type $\mathsf A_1\times \mathsf A_2$. The subgroup $K/Q^r$ of $Q/Q^r$ is reductive of semi-simple type $\mathsf A_1\times\mathsf A_1$. The simple $L_Q$-module $V_{L_Q}(-\alpha_2)$ does not remain simple under the action of $L_K$: as $L_K$-modules 
\[V_{L_Q}(-\alpha_{2p+q+2})\cong W_1\oplus W_2\oplus W_3,\]
where $W_1$, $W_2$ and $W_3$ are simple of dimension 6, 4 and 2, respectively. We take $H=H^u\,L$ with $L=L_K$ and $\Lie H^u$ equal to the $L$-complementary of $W_2$ in $\Lie Q^u$.

\subsubsection{}

To describe the last case, $\mathsf T$-21, we do not use its quotient of type ($\mathscr L$), let us consider the following quotient, which is not minimal and is not the composition of quotients of type ($\mathscr L$).
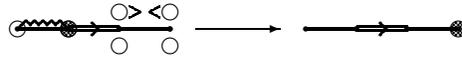
\[\begin{picture}(16200,1800)
\put(300,900){
\put(0,0){\usebox{\atwo}}
\put(1800,0){\usebox{\btwo}}
\put(3600,0){\usebox{\edge}}
\multiput(3600,0)(1800,0){2}{\usebox{\aone}}
\put(3600,600){\usebox{\toe}}
\put(5400,600){\usebox{\tow}}
}
\put(6600,900){\vector(1,0){3000}}
\put(10500,900){
\put(0,0){\usebox{\dynkinf}}\put(5400,0){\usebox{\gcircle}}
}
\end{picture}\]
Here $G$ is of type $\mathsf F_4$, and $H$ is the parabolic subgroup of semi-simple type $\mathsf B_2$ of the symmetric subgroup of type $\mathsf B_4$.

\section{Primitive positive 1-combs}\label{sect:pp1c}
Let $\s=(S^p,\Sigma,\A)$ be a spherical system with a primitive positive 1-comb $D\in\A$. The quotient $\s\to\s/\{D\}$ is of type $(\mathscr L)$, and if $\s$ has rank $>1$ then the quotient is non-essential. Nevertheless, the approach of Section~\ref{subsect:noness} is not very convenient here, and we give a different argument.

In general, by \cite[\S3.6.1]{Lu01}, morphisms between wonderful varieties corresponding to quotients by subsets of colors of the form $\{D\}$ where $D$ is a positive comb consist of projective fibrations: smooth (surjective) morphisms with fibers isomorphic to projective spaces. In particular the following is known.

\begin{proposition}[{\cite[Proposition~3.6]{Lu01}}]
Let $\s=(S^p,\Sigma,\A)$ be a spherical system with a positive comb $D\in\A$ such that $S_D\cap\supp(\Sigma\smallsetminus S)=\varnothing$. Then the geometric realizability of $\s$ follows from the geometric realizability of $\s/\{D\}$.
\end{proposition}

Therefore, we can restrict here to rank $>2$ spherical systems with a primitive positive 1-comb $D$ such that $S_D\cap\supp(\Sigma\smallsetminus S)\neq\varnothing$. There are only 4 such spherical systems.

For each such system $\s$ we give in the next sections the corresponding general isotropy $H$. In each case $H$ is spherical and self-normalizing. To check that the corresponding wonderful variety $X$ has spherical system equal to $\s$ we first notice that
$\s_X$ has the same minimal quotient we indicate for $\s$, this is obvious from our description of $H$. Then there is only one possible couple $(\rank X, S^p_X)$ compatible with the dimension of $X$ (see the formula of \cite[\S5.2]{Lu01}), with the defect of $\s_X$ and with its minimal quotient we have given.

It follows that $\s_X$ and $\s$ have the same rank. Finally, defect, rank and the minimal quotient we have given determine $\s$ uniquely, which implies $\s_X=\s$.

\subsection{}
Here $G$ is of type $\mathsf B_n$. We describe the subgroup $H$ in case of even $n$ (the odd case is slightly more complicated but follows by localization). Let us consider the following quotients $\s\to\s_1\to\s_2$ of type ($\mathscr L$).
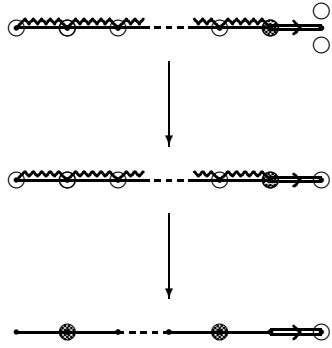
\[\begin{picture}(11400,12600)
\put(300,11700){
\put(0,0){\usebox{\atwo}}
\put(1800,0){\usebox{\atwoseq}}
\put(9000,0){\usebox{\btwo}}
\put(10800,0){\usebox{\aone}}
}
\put(5700,10500){\vector(0,-1){3000}}
\put(300,6300){
\put(0,0){\usebox{\atwo}}
\put(1800,0){\usebox{\atwoseq}}
\put(9000,0){\usebox{\btwo}}
\put(10800,0){\usebox{\wcircle}}
}
\put(5700,5100){\vector(0,-1){3000}}
\put(300,900){
\multiput(0,0)(5400,0){2}{\usebox{\dthree}}
\put(3600,0){\usebox{\shortsusp}}
\put(9000,0){\usebox{\rightbiedge}}
\put(10800,0){\usebox{\wcircle}}
}
\end{picture}\]
Let $Q$ be the parabolic subgroup containing $B_-$ corresponding to $S\smallsetminus\{\alpha_n\}$. The subgroup $K_2$ corresponding to $\s_2$ contains $Q^r$ and $K_2/Q^r$ is very reductive of type $\mathsf C_{n/2}$ in $Q/Q^r$ which is of type $\mathsf A_n$. As $L_Q$-modules
\[\Lie Q^u\cong V(-\alpha_n)\oplus [\Lie Q^u,\Lie Q^u]\cong V(-\alpha_n)\oplus V(-\alpha_{n-1}-2\alpha_n),\]
and $V(-\alpha_n)$ remains simple under the action of $L_{K_2}$.
The subgroup $K_1$ corresponding to $\s_1$ has $L_{K_1}=L_{K_2}$ and $\Lie K_1^u=[\Lie Q^u,\Lie Q^u]$. As $L_{K_2}$-modules
\[[\Lie Q^u,\Lie Q^u]\cong V(-\alpha_{n-1}-2\alpha_n)\oplus V(0).\]
The subgroup $H$ corresponding to $\s$ has $L=L_{K_2}$ and $\Lie H^u$ of codimension 1 in $[\Lie Q^u,\Lie Q^u]$.

\subsection{}
Let us consider the following minimal quotients $\s\to\s_1\to\s_2$: the second one, $\s_1\to\s_2$, is of type ($\mathscr L$) while the first one, $\s\to\s_1$, is not.
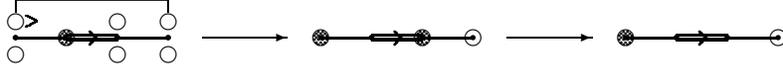
\[\begin{picture}(27600,2250)
\put(300,900){
\multiput(0,0)(3600,0){2}{\usebox{\edge}}
\multiput(0,0)(5400,0){2}{\usebox{\aone}}
\put(3600,0){\usebox{\aone}}
\put(1800,0){\usebox{\btwo}}
\multiput(0,900)(5400,0){2}{\line(0,1){450}}
\put(0,1350){\line(1,0){5400}}
\put(0,600){\usebox{\toe}}
}
\put(6900,900){\vector(1,0){3000}}
\put(10800,-450){
\put(300,1350){\usebox{\dynkinf}}
\multiput(300,1350)(3600,0){2}{\usebox{\gcircle}}
\put(5700,1350){\usebox{\wcircle}}
}
\put(17700,900){\vector(1,0){3000}}
\put(21600,-450){
\put(300,1350){\usebox{\dynkinf}}
\put(300,1350){\usebox{\gcircle}}
\put(5700,1350){\usebox{\wcircle}}
}
\end{picture}\]
Here $G$ is of type $\mathsf F_4$. Let $Q$ be the parabolic subgroup containing $B_-$ corresponding to $S\smallsetminus\{\alpha_4\}$. The subgroup $K_2$ corresponding to $\s_2$ contains $Q^r$ and $K_2/Q^r$ is very reductive of type $\mathsf A_3$ (or equivalently $\mathsf D_3$) in $Q/Q^r$ which is of type $\mathsf B_3$. To be more precise, the root subsystem of $K_2/Q^r$ is generated by $\{\alpha_1,\alpha_2,\alpha_2+2\alpha_3\}$. As $L_Q$-modules
\[\Lie Q^u\cong V(-\alpha_4)\oplus [\Lie Q^u,\Lie Q^u],\]
and the 8-dimensional $L_Q$-module $V(-\alpha_4)$ decomposes into two 4-dimensional $L_{K_2}$-submodules.
The subgroup $K_1$ corresponding to $\s_1$ has $L_{K_1}=L_{K_2}$ and $\Lie K_1^u$ as the $L_{K_2}$-complementary in $\Lie Q^u$ of the $L_{K_2}$-simple submodule $V(-\alpha_4)$. 
The subgroup $H$ corresponding to $\s$ is the parabolic subgroup of $K_2$ containing $B_-\cap K_2$ corresponding to $\{\alpha_1,\alpha_2\}$.

\subsection{}
Let us consider the following minimal quotient $\s\to\s/\Delta'$ which is not of type ($\mathscr L$).
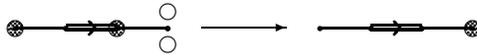
\[\begin{picture}(16800,1800)
\put(300,900){
\put(0,0){\usebox{\dynkinf}}
\multiput(0,0)(3600,0){2}{\usebox{\gcircle}}
\put(5400,0){\usebox{\aone}}
}
\put(6900,900){\vector(1,0){3000}}
\put(11100,900){
\put(0,0){\usebox{\dynkinf}}
\put(5400,0){\usebox{\gcircle}}
}
\end{picture}\]
The subgroup $K$ corresponding to $\s/\Delta'$ is the symmetric subgroup of type $\mathsf B_4$ of $G$, which is of type $\mathsf F_4$. The subgroup $H$ corresponding to $\s$ is the parabolic subgroup of $K$ of semi-simple type $\mathsf A_3$.

\subsection{}
Let us consider the following quotients of type ($\mathscr L$), $\s\to\s_1\to\s_2$.
\[\begin{picture}(27600,1800)
\put(300,900){
\put(0,0){\usebox{\atwo}}
\put(1800,0){\usebox{\btwo}}
\put(3600,0){\usebox{\edge}}
\multiput(3600,0)(1800,0){2}{\usebox{\aone}}
\put(5400,600){\usebox{\tow}}
}
\put(6900,900){\vector(1,0){3000}}
\put(10800,900){
\put(0,0){\usebox{\atwo}}
\put(1800,0){\usebox{\btwo}}
\put(3600,0){\usebox{\edge}}
\put(3600,0){\usebox{\wcircle}}
\put(5400,0){\usebox{\aone}}
}
\put(17700,900){\vector(1,0){3000}}
\put(21600,900){
\put(0,0){\usebox{\atwo}}
\put(1800,0){\usebox{\btwo}}
\put(3600,0){\usebox{\edge}}
\multiput(3600,0)(1800,0){2}{\usebox{\wcircle}}
}
\end{picture}\]
Here $G$ is still of type $\mathsf F_4$. Let $P$ be the parabolic subgroup containing $B_-$ corresponding to $\{\alpha_2\}$. The subgroup corresponding to $\s_2$ is $K_2=K_2^u\,L$ with $L$ that differs from $L_P$ only by its connected center and $\Lie K_2^u$ that consists of an $L$-complementary in $\Lie P^u$ of an $L$-submodule $W_2$ diagonally embedded in $V(-\alpha_1)\oplus V(-\alpha_3)$.
The subgroup corresponding to $\s_1$ is $K_1=K_1^u\,L$ where $\Lie K_1^u$ is the $L$-complementary in $\Lie K_2^u$ of $V(-\alpha_4)$.
The $L$-submodule $W_1=[W_2,V(-\alpha_3-\alpha_4)]$ of $\Lie K_1^u$ is diagonally embedded in $V(-\alpha_1-\alpha_2-\alpha_3-\alpha_4)\oplus V(-\alpha_2-2\alpha_3-\alpha_4)$.
The subgroup corresponding to $\s$ is $H=H^u\,L$ with $\Lie H^u$ the $L$-complementary in $\Lie K_1^u$ of $W_1$, containing $[\Lie K_1^u,\Lie K_1^u]$.

\end{document}